\documentclass[10pt,twoside,final]{camc}
\usepackage{enumerate}
\usepackage{placeins}
\usepackage{siunitx}
\DeclareSymbolFont{igrfm}{OML}{cmm}{m}{it}
\def\camcsym#1#2{\DeclareMathSymbol{#1}{\mathord}{igrfm}{"#2}}
\camcsym{\alpha}{0B}\camcsym{\beta}{0C}\camcsym{\gamma}{0D}\camcsym{\delta}{0E}
\camcsym{\epsilon}{0F}\camcsym{\zeta}{10}\camcsym{\eta}{11}\camcsym{\theta}{12}
\camcsym{\iota}{13}\camcsym{\kappa}{14}\camcsym{\lambda}{15}\camcsym{\mu}{16}
\camcsym{\nu}{17}\camcsym{\xi}{18}\camcsym{\pi}{19}\camcsym{\rho}{1A}
\camcsym{\sigma}{1B}\camcsym{\tau}{1C}\camcsym{\upsilon}{1D}\camcsym{\phi}{1E}
\camcsym{\chi}{1F}\camcsym{\psi}{20}\camcsym{\omega}{21}\camcsym{\varepsilon}{22}
\camcsym{\vartheta}{23}\camcsym{\varpi}{24}\let\varrho\rho\let\varsigma\sigma
\camcsym{\varphi}{27}\camcsym{\Gamma}{00}\camcsym{\Delta}{01}\camcsym{\Theta}{02}
\camcsym{\Lambda}{03}\camcsym{\Xi}{04}\camcsym{\Pi}{05}\camcsym{\Sigma}{06}
\camcsym{\Upsilon}{07}\camcsym{\Phi}{08}\camcsym{\Psi}{09}\camcsym{\Omega}{0A}
\let\le\leqslant\let\leq\leqslant\let\ge\geqslant\let\geq\geqslant
\let\emptyset\varnothing
\newcommand{\edmt}[3][\footnotesize]{
\par\noindent{#1\sffamily\bfseries #2~}{#1\rmfamily #3}\vskip.6\baselineskip}

\def\funding#1{\edmt{Funding}{#1}}
\def\dataava#1{\edmt{Data Availability}{#1}}
\def\complia#1{\edmt[]{Compliance with Ethical Standards}{}\edmt{Conflict of Interest}{#1}}
\def\ethical#1{\edmt{Ethical Approval}{#1}}
\AtBeginDocument{%
\nocite{label}

}

\newtheorem{Th}{Theorem}[section]
\newtheorem{Lem}[Th]{Lemma}
\newtheorem{Prob}[Th]{Problem}
\newtheorem{Prop}[Th]{Proposition}
\newtheorem{Rem}[Th]{Remark}

\begin{document}
\title{Mathematical and numerical study of symmetry and positivity of the tensor-valued spring constant defined from P1-FEM for two- and three-dimensional linear elasticity}
\author[addressref={aff1}, corref, 
email={ounissioussama@stu.kanazawa-u.ac.jp}]{Oussama Ounissi}%
\author[addressref={aff2},]{Masato Kimura}%
\author[addressref={aff2}]{Hirofumi Notsu}%
\address[id=aff1]{Graduate School of Natural Science and Technology, Kanazawa University, 920-1192, Ishikawa, Japan}%
\address[id=aff2]{Faculty of Mathematics and Physics, Kanazawa University, 920-1192, Ishikawa, Japan}%

\maketitle
\abstract{%
In this study, we consider a spring-block system that approximates a $d$-\!\! dimensional linear elastic body, where $d=2$ or $d=3$. We derive a $d\times d$ matrix as the spring constant using the P1 finite element method with a triangular mesh for the linear elasticity equations. We mathematically analyze the symmetry and positive-definiteness of the spring constant. Even if we assume full symmetry of the elasticity tensor, the symmetry of the matrix obtained as the spring constant is not trivial. However, we have succeeded in proving this in a unified manner for both 2D and 3D cases. This is an alternative proof for the 2D case in Notsu-Kimura (2014) and is a new result for the 3D case. We provide a necessary and sufficient condition for the spring constant to be positive-definite in the case of an isotropic elasticity tensor, along with a sufficient condition in terms of mesh regularity and the Poisson ratio. These theoretical results are supported by several numerical experiments. The positive-definiteness of the spring constant derived from the finite element method plays a vital role in fracture simulations of elastic bodies using the spring-block system.
}
\keywords{spring-block system $\cdot$ linear elasticity  $\cdot$ spring constant $\cdot$ finite element method}
\classification{35Q74 $\cdot$ 65N30  $\cdot$ 74B05  $\cdot$ 74R10}

\section{Introduction}
\setcounter{equation}{0}

Modeling fractures is an endeavor of great interest in several fields of science and engineering. Over the years, that interest has inspired researchers to develop many methods to study cracks, from highly sophisticated mathematical models to natural representations validated through experimentation. Due to its consistency with the theory of linear elasticity and its capability to handle complex domains, the finite element method (FEM) is a popular tool in modeling different problems in structural mechanics. It was extensively adapted with several schemes and used to study fracture mechanics. We specifically mention the extended finite element method (XFEM) and the particle discretization scheme finite element method (PDS-FEM) \cite{LL20, CB22}.

Besides the FEM, other discrete approaches such as the discrete element method (DEM), the rigid body spring model (RBSM), and the applied element method (AEM) have shown great aptitude for modeling fractures \cite{KK20, LX20, BE21, MC20, CA18}. These methods often divide the domain into small elements connected using springs. They offer a natural way of studying crack initiation and propagation by introducing a criterion for the springs to break, thus forming a new crack or propagating an existing one (Fig.~\ref{fig1}).

\begin{figure}[h]
    \centering    \includegraphics[width=0.5\textwidth]{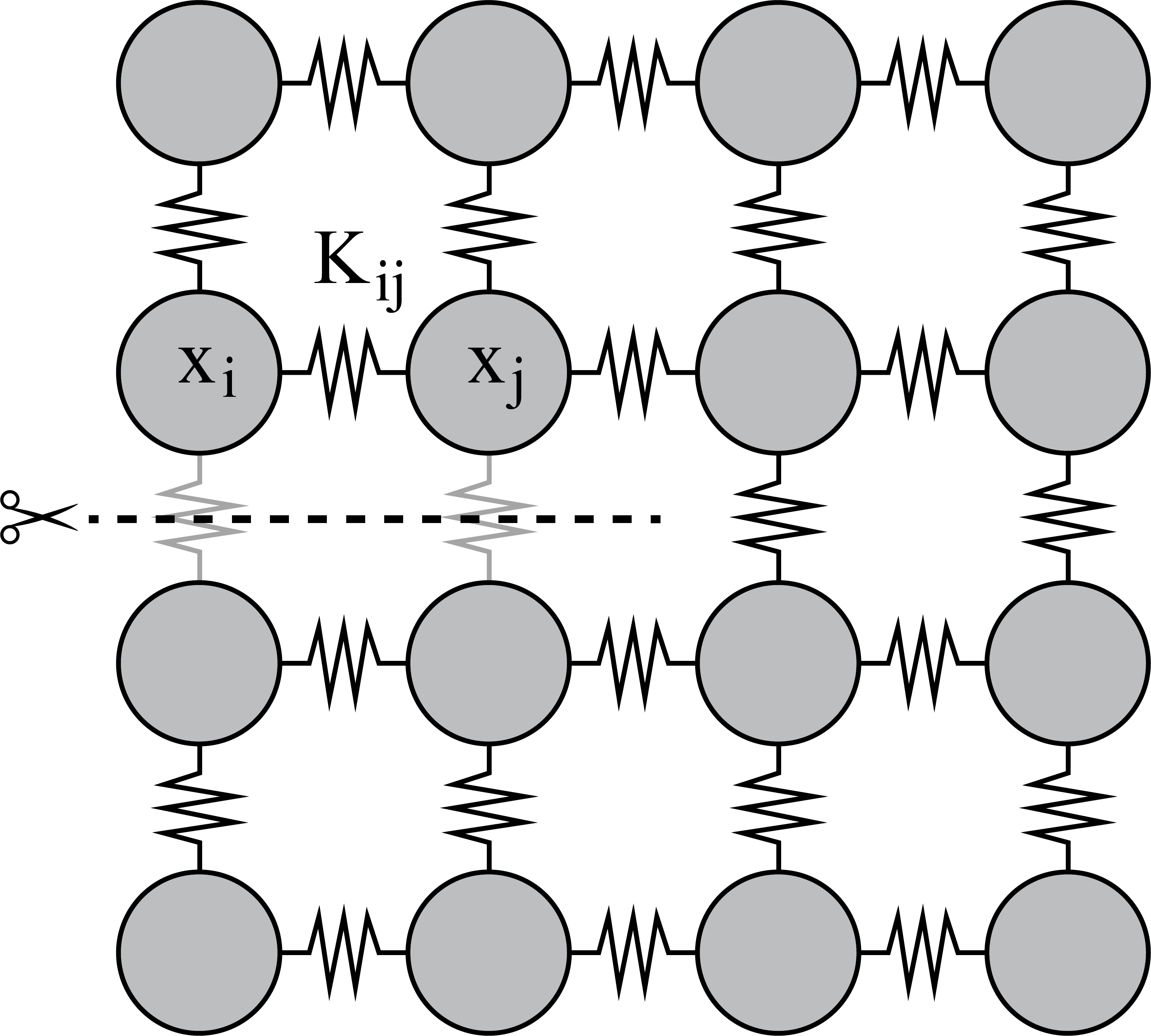}
    \caption{The process of modeling fractures with spring-block methods}
    \label{fig1}
\end{figure}

However, the choice of the spring constants is often made on an ad-hoc basis, meaning the springs are set in a way that results in the most natural simulation in a given scenario. In practice, a popular choice of the spring constants $K_{ij}$ is setting them as scalars or in the form $K_{ij} = k_{ij} I$, where $k_{ij}$ is a scalar and $I$ is the identity matrix. This choice often leads to not representing the Poisson effect \cite{MK05}. Particularly, in the standard lattice spring models, a 2-dimensional model with positive springs can only represent materials with a Poisson's ratio less than $1/4$ for plane strain and $1/3$ for plane stress; in 3D, it is restricted to materials with zero Poisson’s ratio \cite{CHEN14}. This restriction is similar to our findings for a tensor-valued spring constant derived from P1-FEM for linear elasticity. Moreover, in many cases, these models are not described sufficiently mathematically, which leaves much to be desired in terms of mathematical analysis, such as studying their well-posedness. It is also important to note that the consistency of the said models with the theory of linear elasticity is not always guaranteed.

In \cite{KN13}, we proposed a fracture model based on a spring-block system.
The well-posedness of the proposed model was shown by imposing conditions
on the spring constants. In the tensor-valued case, the existence of a unique
solution is guaranteed if all the spring constants are symmetric and positive-definite.
However, the consistency with the linear elasticity theory was still in question.

This question was addressed in \cite{NK14}, and we answered it by deriving a spring-block system model from the P1-FEM discretization for the equation of linear elasticity.
In a loose sense, in this approach, every block corresponds to a nodal point in the FEM mesh, and every spring corresponds to the set of mesh elements that share two nodal points. In particular, a spring constant in 2D is given by a sum of domain integrals over two adjacent triangle elements if at least one of the two nodal points is not on the boundary. In \cite{NK14}, by utilizing the fact above and under the assumption of the homogeneity of the elasticity tensor, the symmetry of the spring constants was proved in the two-dimensional case. Furthermore, a condition was proposed for an isotropic material under which the spring constants are positive-definite in 2D.

We suppose that the elasticity tensor of the elastic material under consideration is given by $C=(c_{pqrs})$ for $p,q,r,s=1,\cdots,d$ with the major symmetry
$c_{pqrs}=c_{rspq}$ and the minor symmetry 
$c_{pqrs}=c_{qprs}=c_{pqsr}$.
We remark that the above symmetry of the spring-constant matrix
corresponds to another symmetry of the form "$c_{pqrs}=c_{rqps}$".
Even in the isotropic case,
the symmetry of the form "$c_{pqrs}=c_{rqps}$" holds only if $\lambda =\mu$, where $\lambda$ and $\mu$ are  the Lam\'{e} constants. So, the symmetry of the spring-constant matrix can not be derived as a simple consequence of the major and minor symmetries of the elasticity tensor.

In this paper, we seek to extend what was carried out in \cite{NK14} to three dimensions. However, the arguments in \cite{NK14} cannot be applied in 3D, since the number of tetrahedral elements corresponding to a given spring is unknown a priori, unlike the two-dimensional case.
To ensure the solvability of the spring-block system derived from P1-FEM, we offer a unified proof for the symmetry of the spring constants
in two and three dimensions under the homogeneity assumption. 
Additionally, under the isotropy assumption, we propose a local condition on the FEM mesh that guarantees the positive-definiteness of the spring constant. 

This paper is organized as follows. In Sections \ref{sec2} and \ref{sec3}, we briefly present a formal representation of a spring-block system problem and show that the P1-FEM formulation of linear elasticity could be interpreted as such. The two sections merely explain the mathematical setting and the previous results, which were studied comprehensively in \cite{KN13} and \cite{NK14}. Sections \ref{sec4} and \ref{sec5} are dedicated to studying the properties of the spring constant derived from P1-FEM for linear elasticity. We start in Section~\ref{sec4} by presenting a unified proof for the symmetry in 2 and 3 dimensions. Then, in Section~\ref{sec5}, we discuss the conditions under which a spring constant is positive-definite. In particular, for isotropic elasticity, we give a necessary and sufficient condition for the positive-definiteness of the spring constant, and also a sufficient condition in terms of mesh regularity and the Poisson ratio. In Section~\ref{sec6}, we offer a numerical examination of the spring constants derived from some sample meshes. Section~\ref{sec7} serves as a conclusion.

\newpage 

\section{spring-block system}\label{sec2}
\setcounter{equation}{0}

\subsection{Block division}

Let $d = 2$ or $3$ and let $\Omega \subset \mathbb{R}^d$ be a $d$-dimensional bounded Lipschitz domain. We introduce a block division $\mathcal{D} = \{ D_i \}_{i = 1}^N$ of $\Omega$, that satisfies $ N  \in \mathbb{N}$, 
$\overline{\Omega} =\cup_{i \in \{1,\dots,N\}} \overline{D_i}$, $D_i \cap D_j = \emptyset$ $(i\neq j)$,
where each subblock $D_i$ is a non-empty connected open subset of $\mathbb{R}^d$ with a Lipschitz boundary. For $i,j \in \{1,\dots,N\}$ with $i\neq j$, we define $d_{ij}:=\mathcal{H}^{d-1}(\overline{D_i} \cap \overline{D_j})$,
where $\mathcal{H}^{d-1}$ is the $(d-1)$-Hausdorff measure, that means, the length for $d=2$ and the area for $d=3$. 
For every subblock $D_i \in \mathcal{D}$, we define the set of indices of all its neighboring subblocks by
$\Lambda_i := \{ j \in \{1,\dots, N \}\setminus\{i\} ,~d_{ij} > 0\}$.
We also define the set of all pairs of indices corresponding to two neighboring subblocks as   
$\Lambda := \{ (i,j); 1 \leq i < j \leq N, ~d_{ij} > 0  \}$.
    
\subsection{spring-block system formulation}
This paper focuses on the vector-valued displacements, vector-valued forces, and tensor-valued (matrix-valued) spring constants since the scalar case has already been studied thoroughly in \cite{NK14}.

For a block division $\mathcal{D}$ of $\Omega$, we consider a piecewise constant displacement field $u = \sum_{i=1}^N u_i \chi_i \in V(\mathcal{D})^d$, where
\begin{align*}
\chi_i(x) &:= \left\{\begin{array}{l}
          1 \quad x \in D_i \\
          0 \quad x \in \Omega \backslash D_i
    \end{array}
    \right. \quad (i=1,\dots,N), \\
    V(\mathcal{D})^d &:= \bigl\{ v \in L^\infty(\Omega); v = \sum_{i =1}^N  v_i\chi_i,~ v_i \in \mathbb{R}^d \bigr\}.
\end{align*}

For each $(i,j) \in \Lambda$, we consider a virtual spring between an adjacent pair of subblocks $D_i$ and $D_j$, and we denote its spring constant by $K_{ij}\in \mathbb{R}^{d\times d}$.
We suppose that the virtual spring obeys the generalized Hooke's law. Meaning, the force from $D_j$ acting on $D_i$ 
is assumed to be given by $K_{ij}(u_j-u_i)$. From the action-reaction law, we impose that $K_{ij} = K_{ji}$. We denote
\begin{equation}\label{def-K}
K=\{K_{ij}\}_{(i,j)\in\Lambda}~\text{with}~K_{ij} = K_{ji}\in \mathbb{R}^{d\times d}.
\end{equation}
For a block division $\mathcal{D}$ of $\Omega$ and a set of spring constants $K$ with \eqref{def-K},
we call $(\mathcal{D},K)$ a tensor-valued spring-block system.

When considering a boundary value problem of linear elasticity, we often need to set a Dirichlet boundary condition. We suppose $J_1$ is a non-empty set of indices corresponding to the subblocks where the displacement is given a priori. The force balance is considered at the remaining subblocks with corresponding indices in $ J_0$. Hereafter, we denote $J = (J_0, J_1)$ under the assumptions:
\begin{equation*}
   J_0\cup J_1 = \{1,\dots,N\},\quad
   J_0\cap J_1 =\emptyset, \quad J_0 \neq \emptyset, \quad J_1 \neq \emptyset.
\end{equation*}
We also define
\begin{equation*}
    V_\ell(\mathcal{D})^d := \bigl\{ v \in L^\infty(\Omega); v = \sum_{i \in J_\ell}  v_i\chi_i, \quad v_i \in \mathbb{R}^d \bigr\}\quad (\ell =0,1).
\end{equation*}
Similarly, we call $(\mathcal{D}, K, J)$ a tensor-valued spring-block system with a Dirichlet boundary.
\begin{Prob}\label{prob1}
Let $(\mathcal{D}, K, J)$ be a tensor-valued spring-block system with a Dirichlet boundary. For a given body force $F = \sum_{i \in J_{0}} F_i \chi_i \in V_0(\mathcal{D})^d$, and a given displacement $g = \sum_{i \in J_{1}} g_i \chi_i \in V_1(\mathcal{D})^d$, find a displacement $u = \sum_{i = 1}^N u_i \chi_i \in V(\mathcal{D})^d$ such that
\begin{align}\label{eq:prob1}
\begin{cases}
    \displaystyle{\sum_{j\in \Lambda_i} K_{ij}(u_j-u_i) + F_i = 0}&(i \in J_0),\\
    u_i= g_i, & (i \in J_1). 
\end{cases}
\end{align}
\end{Prob}
We introduce the following bilinear form and elastic energy for Problem~\ref{prob1}
\begin{align*}
    (u,v)_K &:= \sum_{(i,j) \in \Lambda} \{ K_{ij}(u_j-u_i)\}\cdot (v_j-v_i)\quad (u,v \in V(\mathcal{D})^d),\\
    E_{\rm el}(u) &:= \frac{1}{2} (u,u)_K - \sum_{i \in J_0} F_i \cdot u_i
    \quad (u\in V(\mathcal{D})^d).
\end{align*}
In \cite{KN13, NK14}, the solvability of Problem~\ref{prob1} is discussed extensively. It was shown that there exists a unique solution to Problem~\ref{prob1}, if every block $(D_i, i\in J_0)$ is connected to a block $(D_j, j \in J_1)$ by a chain of springs with positive-definite spring constants (see \cite{KN13, NK14} for the precise definition). Moreover, the use of the spring-block system for fracture mechanics is usually done by cutting the spring connecting two subblocks when the strain or the damage reaches a certain threshold. If the said spring has a constant that is not nonnegative-definite, the energy dissipation of the model may not hold. Accordingly, both the solvability and energy dissipation properties are guaranteed if $K_{ij} \in \mathbb{R}^{d \times d}_{\rm sym}$ and $K_{ij}>O$ for all $(i,j)\in\Lambda$.

\section{P1-FEM based spring-block system for linear elasticity}\label{sec3}
\setcounter{equation}{0}
Based on the idea of \cite{NK14}, we consider a P1-FEM for a linear elasticity problem with a triangular mesh on $\Omega$
to construct a spring-block system that is consistent with the given linear elasticity system.

Let $\Omega \subset \mathbb{R}^d$ be a domain with a Lipschitz
boundary denoted by $\Gamma$.
For simplicity, we suppose that $\Gamma$ is a polygon for 2D or a polyhedron for 3D. We assume that $g \in C^0(\overline\Omega)^d \cap H^1(\Omega)^d $, $f \in L^2(\Omega)^d$, and $c_{pqrs} \in L^\infty(\Omega)$ for $\ p,q,r,s \in \{1,\dots, d\})$ are given functions. Here, and in the following, we denote a function space valued in ${\mathbb R}^d$, such as $H^1(\Omega;{\mathbb R}^d)$, by $H^1(\Omega)^d$.

We assume that the elasticity tensor $C=(c_{pqrs})(x)$ satisfies the following symmetries, the major symmetry
$c_{pqrs}=c_{rspq}$ and the minor symmetry 
$c_{pqrs}=c_{qprs}=c_{pqsr}$ for $p,q,r,s\in\{1,\cdots,d\}$.
In case the elasticity tensor satisfies the major and minor symmetries, it is said that it satisfies the full symmetry. We additionally assume that $C=(c_{pqrs}(x))$ is uniformly positive-definite; i.e., there exists $c_*>0$ such that
\begin{align*}
c_{pqrs}(x)\xi_{pq}\xi_{rs}\ge c_* \vert \xi\vert^2
\quad \xi\in \mathbb{R}^{d \times d}_{\rm sym},
~\mbox{a.e.}~x\in\Omega,
\end{align*}
where and hereafter, we use the Einstein notation for the summation.

For a displacement $u =(u_1,\cdots,u_d)^T\in H^1(\Omega)^d$, 
we define the strain tensor $e[u](x)=(e_{pq}[u](x))\in \mathbb{R}^{d \times d}_{\rm sym}$, and the stress tensor 
$\sigma [u](x)=(\sigma_{pq}[u](x))\in \mathbb{R}^{d \times d}_{\rm sym}$
by
\begin{align*}
e_{pq}[u] := \frac{1}{2}
    \left( u_{p,q} + u_{q,p}\right),\quad
\sigma_{pq}[u] := c_{pqrs} e_{rs}[u],
\end{align*}
where $\varphi_{,p}$ represents the partial derivative of a function $\varphi(x)=\varphi (x_1,\cdots,x_d)$ with respect to the variable $x_p$.

We consider the following boundary value problem of linear elasticity.
\begin{align}\label{ep}
\begin{cases}
-\mbox{\rm div}\sigma[u] =f\quad &\mbox{in}~\Omega,\\
u=g\quad &\mbox{on}~\Gamma,
\end{cases}
\end{align}
where $(\mbox{\rm div}\sigma)_p=\sigma_{pq,q}$
for $p=1,\cdots,d$.

We define the bilinear form:
\begin{align*}
    a(u,v) &= \int_\Omega \sigma(u):e(u)dx \quad
    (u,v\in H^1(\Omega)^d),
\end{align*}
and the affine function space $V(g)$ as
\begin{equation*}
    V(g) := \{ v \in H^1(\Omega)^d;~v = g \ {\rm on} \ \Gamma \}.
\end{equation*}
We also denote the innerproduct of $L^2(\Omega)^d$ by $\langle \cdot,\cdot\rangle$.
We consider the weak formulation of \eqref{ep} as follows.
\begin{Prob}\label{prob2}
Find $u \in V(g)$ that satisfies
\begin{equation*}\label{ewf}
    a(u,v) = \langle f,v \rangle \quad
    (v \in V(0)).
\end{equation*}
\end{Prob}

We consider a piecewise linear finite element method on a triangular mesh of $\Omega$. The triangulation is denoted by $\mathcal{T}_h = 
\{T^\alpha\}_{\alpha=1}^{N_e}$, where $T^\alpha$ is a triangular element (a closed triangle or tetrahedron including the interior and boundary) and $N_e \in \mathbb{N}$ is the number of the triangular elements.
We denote the set of nodal points of $\mathcal{T}_h$ by $\mathcal{N}_h = \{P_j\}_{j=1}^{N_p}$, 
where $N_p \in \mathbb{N}$ is the number of nodal points.
We define $J=(J_0,J_1)$ as
\begin{equation}\label{FEMJ0J1}
J_0 := \{ j \in \{1,\dots,N_p\};~P_j \in \Omega \},\quad 
J_1 := \{ j \in \{1,\dots,N_p\};~P_j \in \Gamma \}.
\end{equation}

We also define the function space of FEM by 
\begin{align*}
P_1(\mathcal{T}_h) :=\{
v_h\in C^0(\overline\Omega);~v_h\vert_{T^\alpha}
\mbox{ is linear}~(\alpha=1,\cdots,N_e)\},
\end{align*}
and a basis of $P_1(\mathcal{T}_h)$ by
$\langle \varphi_1,..,\varphi_{N_p} \rangle$, where $\varphi_j\in P_1(\mathcal{T}_h)$ with $\varphi_j(P_i) = \delta_{ij} $.
We denote the vector-valued finite element space 
by $P_1(\mathcal{T}_h)^d$, and also introduce
an affine subspace of $P_1(\mathcal{T}_h)^d$, in which we will seek a finite element solution, as
\begin{equation*}
    V_h(g) := \{ v_h \in P_1(\mathcal{T}_h)^d;
    ~v_h (P_i)= g(P_i)~(i\in J_1) \}.
\end{equation*} 
A standard P1-FEM for Problem~\ref{prob2} is defined as follows:
\begin{Prob}\label{prob3}
Find $u_h \in V_h(g)$ that satisfies 
\begin{equation*}
    a(u_h,v_h) = \langle f,v_h \rangle\quad (v_h \in V_h(0)).
\end{equation*}
\end{Prob}
From the well-known Poincar\'e and Korn inequalities, it follows that the bilinear form $a(\cdot,\cdot)$ is coercive on $V(0)=H^1_0(\Omega)^d$ \cite{D-L1976,FEM1}. Therefore, by the Lax-Milgram theorem, Problem~\ref{prob2} and Problem~\ref{prob3} have unique solutions, 
and the convergence of $u_h$ to $u$ is shown by Cea's lemma under a certain regularity condition for the mesh refinement \cite{FEM1, FEM2}.

Problem \ref{prob3} is equivalent to finding $u_h \in V_h(g)$ that satisfies 
\begin{equation}\label{feme}
    a(u_h,\varphi_i e_k) = \langle f,\varphi_i e_k \rangle \quad 
    (i \in J_0,~k \in \{1, \dots, d\}),
\end{equation}
where $\{e_k\}_{k = 1}^d$ is the standard basis of $\mathbb{R}^d$.

The next proposition states that the finite element method given by Problem~\ref{prob3} is interpreted as a spring system.
We set $N:=N_p$, and, for $i,j\in \{1,\cdots,N\}$, we define $K_{ij}=(K_{ij}^{kl})_{k,l=1,\cdots,d}\in \mathbb{R}^{d \times d}$ by
\begin{equation}\label{K-FEM}
    K_{ij}^{kl} := -a(\varphi_j e_l,\varphi_i e_k).
\end{equation}
We remark that $K_{ij}=K_{ji}^T$ holds, since the bilinear
form $a(\cdot,\cdot)$ is symmetric:
\begin{align*}
  K_{ij}^{kl}=-a(\varphi_j e_l,\varphi_i e_k)
  = - a(\varphi_i e_k,\varphi_j e_l)=K_{ji}^{lk}.
\end{align*}

We set 
\begin{align}\label{Lambda-FEM}
\Lambda_i:=\{j\in \{1,\cdots, N\}\setminus \{i\};\,K_{ij}\neq O\}
\quad (i\in\{1,\cdots, N\}) .
\end{align}
We also define
$F_i=(F_i^k)_{k=1,\dots,d} \in \mathbb{R}^d$,
and $g_i \in \mathbb{R}^d$ by
\begin{equation}\label{Fg}
F_{i}^k := \langle f,\varphi_i e_k \rangle ~(i\in J_0),
    \quad g_i := g(P_i)~(i\in J_1).
\end{equation}
\begin{Prop}\label{equiv}
Under the settings $u_h = \sum_{i=1}^{N_p}u_i\varphi_i\in P_1(\mathcal{T}_h)^d$, \eqref{FEMJ0J1}, \eqref{K-FEM}, \eqref{Lambda-FEM}, and \eqref{Fg},
$u_h$ is a solution to Problem~\ref{prob3}
if and only if \eqref{eq:prob1} holds.
\end{Prop}
\begin{proof}
Since $\sum_{j=1}^{N_p}\varphi_j=1$ and 
\begin{align*}
  0=a(e_l, \varphi_ie_k)=\sum_{j=1}^{N_p}a(\varphi_je_l, \varphi_ie_k)
  =-\sum_{j=1}^{N_p}K_{ij}^{kl},
\end{align*}
we obtain
\begin{align}\label{Kiij}
  K_{ii}=-\sum_{j\neq i}K_{ij}=-\sum_{j\in \Lambda_i}K_{ij}\quad (i\in \{1,\cdots,N_p\}).
\end{align}

For $i\in \{1,\cdots,N_p\}$, setting $u_i^l:=u_i\cdot e_l$, we have $u_i=\sum_{l=1}^du_i^le_l$.
We write $u_h$ as
\begin{align*}
u_h=\sum_{j=1}^{N_p}\sum_{l=1}^du_j^l\varphi_ie_l.
\end{align*}

For any $i \in J_0$ and $k \in \{1, \dots, d\}$,
applying the equality \eqref{Kiij}, we have
\begin{align*}
\langle f,\varphi_i e_k \rangle -a(u_h,\varphi_i e_k) 
&=
F_i^k -\sum_{j=1}^{N_p}\sum_{l=1}^du_j^la(\varphi_je_l,\varphi_i e_k)\\
&=
F_i^k +\sum_{j=1}^{N_p}\sum_{l=1}^dK_{ij}^{kl}u_j^l\\
&=e_k\cdot\left(F_i+ \sum_{j=1}^{N_p}K_{ij}u_j\right)\\
&=e_k\cdot\left(F_i+ \sum_{j\neq i}K_{ij}u_j+K_{ii}u_i\right)\\
&=e_k\cdot\left(F_i+ \sum_{j\in \Lambda_i}K_{ij}(u_j-u_i)\right).
\end{align*}
This implies the equivalency between \eqref{eq:prob1} and Problem~\ref{prob3}.
\end{proof}

\begin{Rem}\label{rem:bd}
To construct a tensor-valued spring-block system with a Dirichlet boundary
$(\mathcal{D}, K, J)$ from Proposition~\ref{equiv}, we have to consider 
an admissible block division, and also have to show the symmetry $K_{ij}=K_{ji}$ and its positivity $K_{ij}> 0$. 

A candidate of such divisions is the barycentric domain based on the FEM mesh shown in Figure~\ref{fig2a} 
(see \cite{FEM1} for the barycentric coordinate). 
For an acute mesh in 2D (i.e., all $T^\alpha \in \mathcal{T}_h$ is an acute triangle), the Voronoi diagram satisfying 
\begin{equation*}
    D_i = \{x\in \Omega; ~ \lvert  x - P_i \rvert \leq \lvert x - P_j \rvert  ~\forall j \in \{1,\dots,N_p\}\backslash \{i\}\}~~(i \in \{1,\dots,N_p\}), 
\end{equation*}
is also admissible (Figure~\ref{fig2b}).

The symmetry of the spring constant will be shown in 
Theorem~\ref{th1}. We will also find that the positivity 
holds under a certain condition 
but not always true (Theorem~\ref{th3}).
\end{Rem}
\begin{figure}[ht]
    \centering
    \subfloat[Barycentric]{\label{fig2a}\includegraphics[width=.48\linewidth]{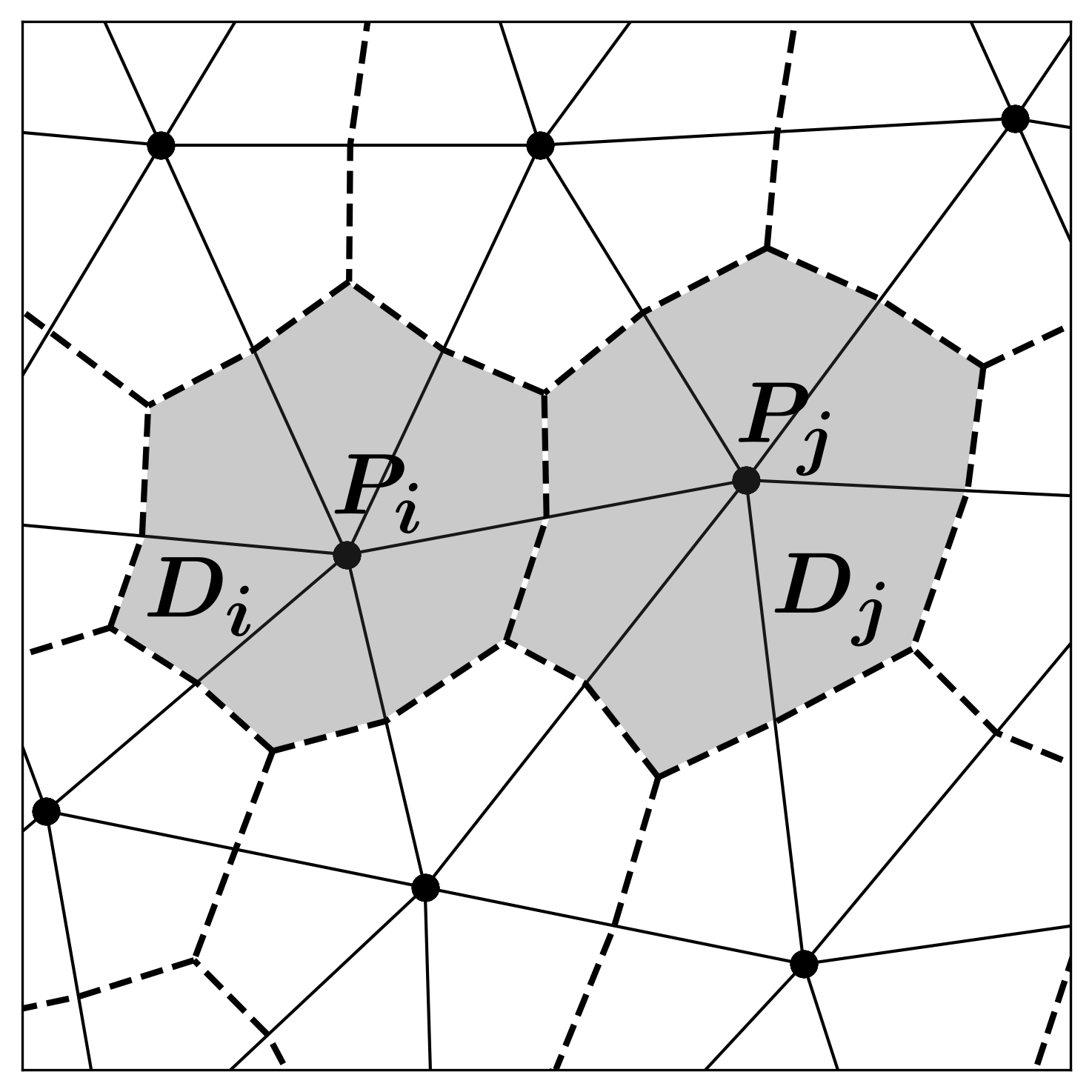}}\quad
    \subfloat[Voronoi]{\label{fig2b}\includegraphics[width=.48\linewidth]{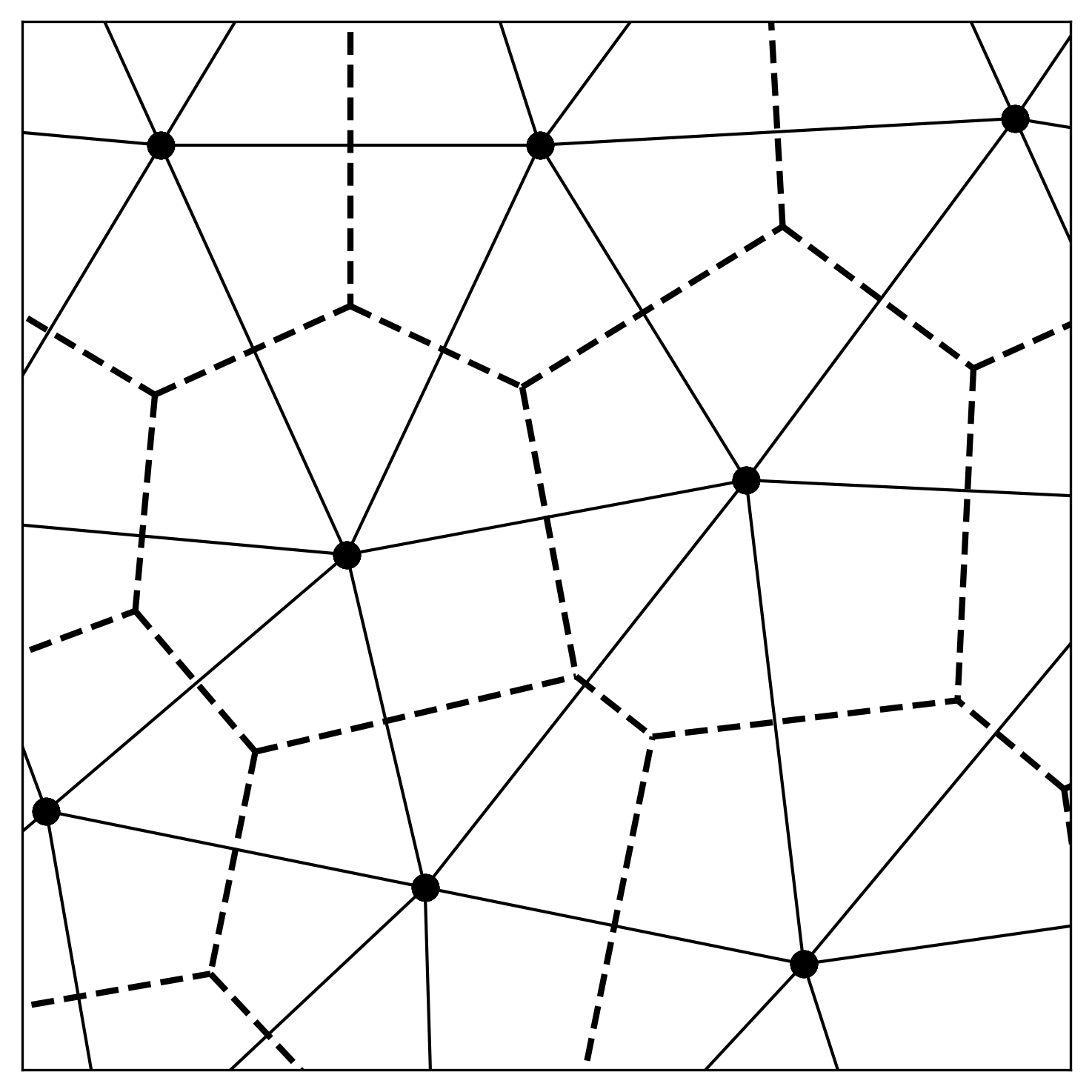}}
    \caption{Block divisions based on a FEM triangular mesh}
    \label{fig2}
\end{figure}

\section{Symmetry of the spring constant}\label{sec4}

\setcounter{equation}{0}

In this section, we offer a unified proof of the symmetry of the spring constant derived from the P1-FEM scheme in two and three dimensions, i.e., $K_{ij}\in \mathbb{R}^{d\times d}_{\rm 
sym}$.
In all of the following, we assume the homogeneity of the elasticity tensor. 
Let $\Gamma^\beta$ denote an edge of a triangle $T\in \mathcal{T}_h$ for $d=2$ or a face of 
a tetrahedron $T\in \mathcal{T}_h$ for $d=3$, and assume that $\Gamma^\beta$ is a closed set. Furthermore, define $n^\beta=(n^\beta_1,\cdots,n^\beta_d)^T\in {\mathbb R}^d$ as one of the normal unit vectors on $\Gamma^\beta$.
We denote the set of all the edges/faces on $\mathcal{T}_h$ by $\mathcal{B}_h = \{\Gamma^\beta \}_{\beta=1}^{N_b}, N_b \in \mathbb{N}$. 

For $d = 3$, we use the standard vector cross product $a \times b$ for vectors $a, b \in {\mathbb R}^3$.
For $d = 2$, we define the scalar-valued cross product by
\[
a \times b := a_1 b_2 - a_2 b_1,
\]
which corresponds to the third component of the 3D cross product.
We remark that for nonzero vectors $a$ and $b$, we have $a \times b = 0$ if and only if $a$ is parallel to $b$.

\begin{figure}[ht]
    \centering
    \subfloat[$d = 2$]{\label{fig3a}\includegraphics[width=.48\linewidth]{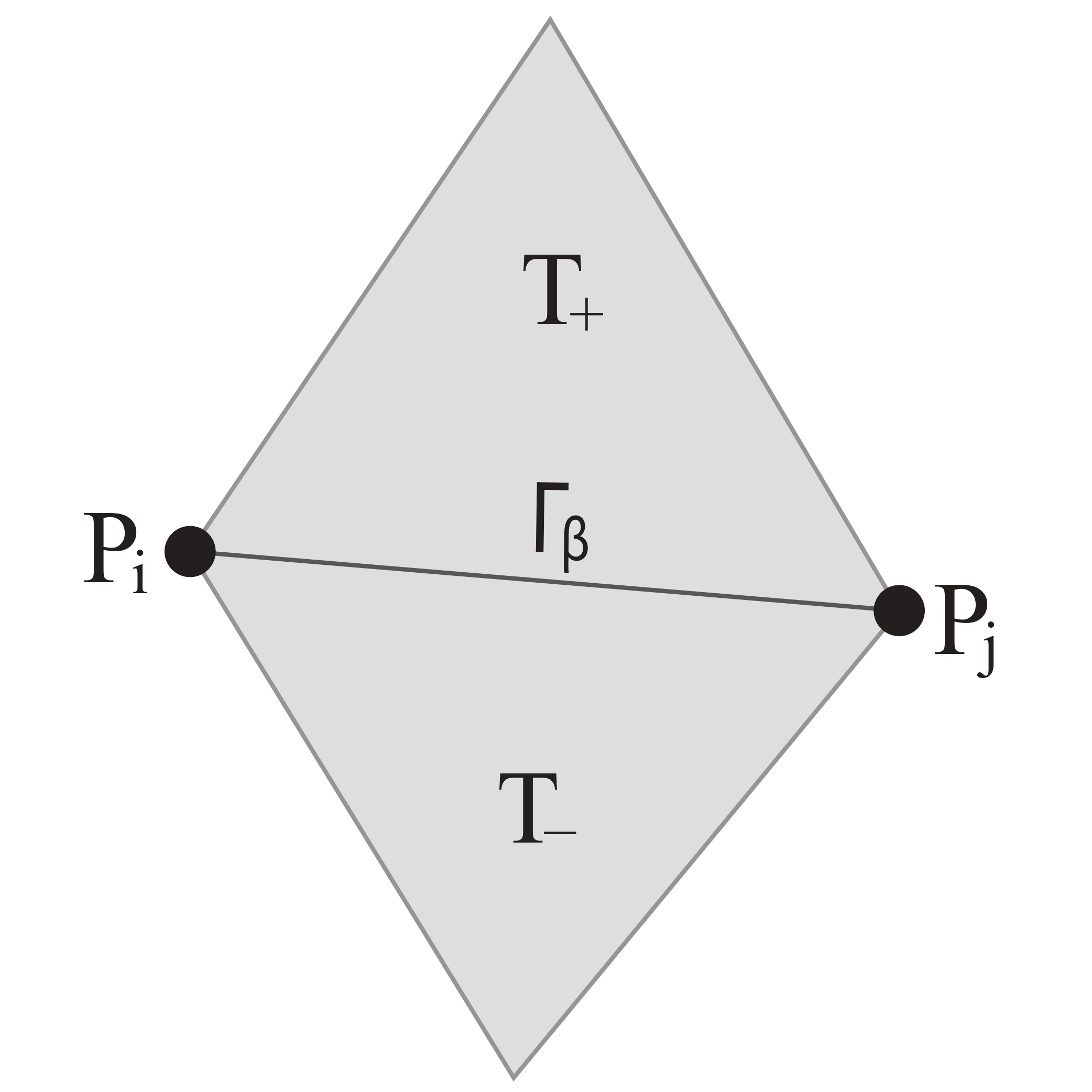}}\quad
    \subfloat[$d = 3$]{\label{fig3b}\includegraphics[width=.48\linewidth]{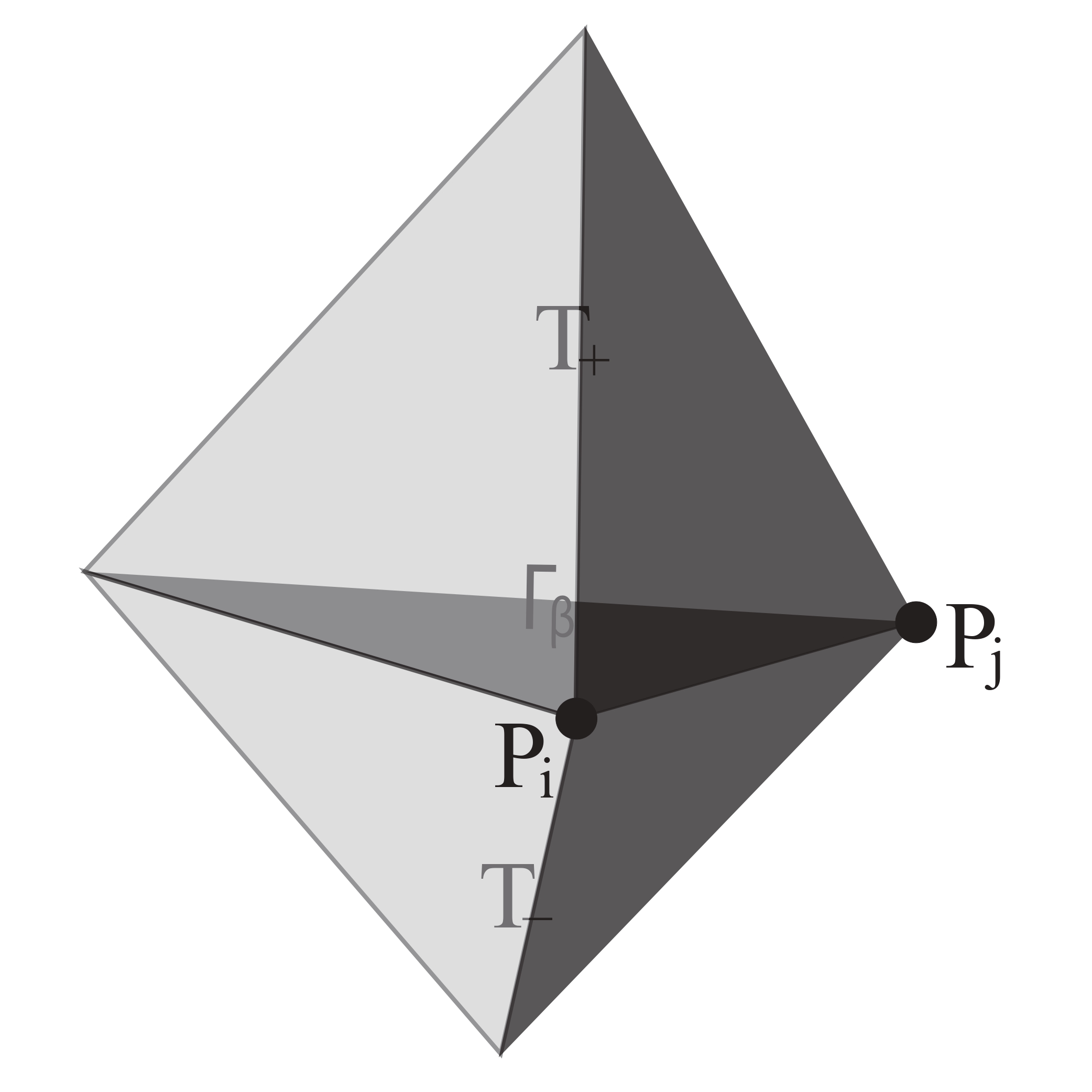}}\par
    \caption{Two adjacent elements $T_\pm$ on a triangular mesh ${\mathcal T}_h$ with $T_+\cap T_-=\Gamma_\beta$:  
    (a) $T_+\cup T_-=\overline{D_{ij}}$ in 2D, (b) $T_+\cup T_-\subset \overline{D_{ij}}$ in 3D.}
    \label{fig3}
\end{figure}

\begin{Lem}\label{lem1}
We suppose $\Gamma^\beta \in \mathcal{B}_h$, such that $T_{\pm} \in \mathcal{T}_h$ and $\Gamma^\beta = T_{+} \cap T_{-}$.
We define a nodal point $P_l$
as $\{P_l\}= \mathcal{N}_h \cap (T_{+}\backslash \Gamma^\beta)$, then it holds that
\begin{align}
&\nabla \varphi_l \vert_{T_+} \times n^\beta = 0,\label{lemeq1}\\
&[\nabla \varphi]_\beta \times n^\beta = 0\quad (\varphi \in P_1(\mathcal{T}_h)),\label{lemeq2}
\end{align}
where $[\nabla \varphi]_\beta := \nabla \varphi \vert_{T_+} - \nabla \varphi \vert_{T_-}$.
In other words,
\begin{align}
&\varphi_{l,s} n^\beta_q - \varphi_{l,q} n^\beta_s = 0, \quad (q,s = 1, \cdots, d),\label{lemeq1'}\\
&[\varphi_{,s}]_\beta n^\beta_q - [\varphi_{,q}]_\beta n^\beta_s = 0, \quad (q,s = 1, \cdots, d,~\varphi \in P_1(\mathcal{T}_h)).\label{lemeq2'}
\end{align}
\end{Lem}
\begin{proof}
Since $\varphi_l=0$ on $\Gamma_\beta$, 
for
any tangent vector $\tau$ on $\Gamma_\beta$,
the directional derivative
of $\varphi$ along $\tau$ vanishes, i.e., 
$\tau\cdot \nabla \varphi_l \vert_{T_+} =0$.
This implies \eqref{lemeq1} and also \eqref{lemeq1'}.

Similarly, for \eqref{lemeq2}, since $\varphi\in P_1(\mathcal{T}_h))\subset C^0(\overline \Omega)$, 
for any tangent vector $\tau$ on $\Gamma_\beta$,
the directional derivatives
of $\varphi\vert_{T_\pm}$ along $\tau$ coincide, i.e., 
$\tau\cdot \nabla \varphi \vert_{T_+} =
\tau\cdot \nabla \varphi \vert_{T_-}$ holds.
This implies \eqref{lemeq2} and also \eqref{lemeq2'}.
\end{proof}

In this section, we suppose 
\begin{align}\label{A-Pij}
\mbox{
$P_i,~P_j\in {\mathcal N}_h$,\quad
$j\in \Lambda_i$ (or equivalently, $i\in \Lambda_j$),\quad $\{P_i, P_j\} \not\subset \Gamma$.}
\end{align}
We introduce the following notation:
\begin{align*}
&D_{ij}  := \text{Supp}(\varphi_i) \cap \text{Supp}(\varphi_j),\\
&\Gamma_i := \{ x \in \partial D_{ij} ; \varphi_i(x) \neq 0 \}, \quad
\Gamma_j  := \{ x \in \partial D_{ij} ; \varphi_j(x) \neq 0 \},
\end{align*}
\begin{align*}
        \overline{D_{ij}} & = \underset{\alpha=1}{\overset{\alpha_0}{\cup}} T_\alpha, \quad \{ T_\alpha \}_{\alpha = 1}^{\alpha_0} \subset \mathcal{T}_h \quad (\alpha_0 = 2 ~\text{when}~ d=2,~ \alpha_0 \geq 3 ~\text{when}~ d=3 ).
\end{align*}
Then, $\partial D_{ij}  = \Gamma_i \cup \Gamma_j$ holds.
For a given $\alpha \in \{ 1, \cdots, \alpha_0 \} $, we denote the outward normal vector on $\partial T_\alpha$ by ${\nu^\alpha}$, and we define the different sections of $\partial T_\alpha$ as follows
\begin{align*}
        \partial_i T_\alpha  := \Gamma_i \cap \partial T_\alpha, \quad
        \partial_j T_\alpha  := \Gamma_j \cap \partial T_\alpha, \\
        \partial_1 T_\alpha := \partial T_\alpha \cap \partial D_{ij},\quad
        \partial_0 T_\alpha  := \partial T_\alpha \backslash \partial_1 T_\alpha.
\end{align*}
It is also useful to state the following relations:
\begin{align*}
    \partial T_\alpha  = \partial_0 T_\alpha \cup \partial_1 T_\alpha, \quad
    \partial_1 T_\alpha  = \partial_i T_\alpha \cup \partial_j T_\alpha.
\end{align*}
We denote the $d-1$ dimensional surface integral by $d\mathrm{S}$ in the following argument.
\begin{Lem}\label{lem3}
Under the above settings, it holds that
\begin{equation*}
    \sum_{\alpha = 1}^{\alpha_0}\int_{\partial_0 T_\alpha}\varphi_{j}({\nu}_q^{{\alpha}}\varphi_{i,s} - {\nu}_s^{{\alpha}}\varphi_{i,q})\,d\mathrm{S} = 0.
\end{equation*}
\end{Lem}
\begin{proof}
We note that $\alpha_0 = 2$ when $d=2$, $\alpha_0 \geq 3$ when $d=3$. 1) If $d = 2$, let $\beta_0 = 1$ and $\Gamma_1 = \partial_0 T_\alpha $, for $\alpha = 1,2$. 2) If $d = 3$, there exists $\{\Gamma_\beta\}_{\beta = 1}^{\alpha_0} \subset \mathcal{B}_h $, such that $\underset{\alpha=1}{\overset{\alpha_0}{\cup}} \partial_0 T_\alpha = \underset{\beta=1}{\overset{\alpha_0}{\cup}} \Gamma_\beta $, and for $\beta \in \{1, \cdots, \alpha_0 \}$, there exist unique, distinct $\alpha, \gamma \in \{ 1, \cdots, \alpha_0\}$,  such that $ \Gamma_\beta = \partial_0 T_\alpha \cap \partial_0 T_\gamma$ (Figure~\ref{fig4}). We set $\beta_0 = \alpha_0$.  

Then, noting that $\nu^\alpha = \pm n^\beta$ on $\Gamma_\beta$, we have 
\begin{align*}
    \sum_{\alpha = 1}^{\alpha_0}\int_{\partial_0 T_\alpha}\varphi_{j}(\nu_q^\alpha\varphi_{i,s} - \nu_s^\alpha\varphi_{i,q})d\mathrm{S} & = \sum_{\beta = 1}^{\beta_0}\int_{\Gamma_\beta} \pm \varphi_{j}(n^\beta_q [\varphi_{i,s}]_\beta - n^{\beta}_s [\varphi_{i,q}]_\beta)d\mathrm{S}\quad \\ 
    & = 0. \quad \text{(by Lemma \ref{lem1})}
\end{align*}
\end{proof}
    
\begin{Th}\label{th1}
Under the assumption \eqref{A-Pij}, $K_{ij} = K_{ji} = K_{ij}^T$ holds.
\end{Th}
\begin{proof}
In \cite{NK14}, without assuming any symmetries of the elasticity modulus $C=(c_{pqrs}(x))$, it is shown that $K_{ij}^{kl}$ for $k,l = 1, \dots,d$ can be written component-wise as follows:
\begin{align}
    K_{ij}^{kl} &= -\frac{1}{4} \sum_{q,s=1}^{d} \int_{\Omega} ( c_{kqls} \varphi_{j,s} \varphi_{i,q} + c_{qkls} \varphi_{j,s} \varphi_{i,q} + c_{kqsl} \varphi_{j,s} \varphi_{i,q} + c_{qksl} \varphi_{j,s} \varphi_{i,q} ) dx\notag\\
    &= \sum_{q,s=1}^{d}
    \int_\Omega C_{qs}^{kl} \varphi_{j,s}\varphi_{i,q}\,dx,\label{kijns}
\end{align}
where we set 
\begin{equation*}
    C_{qs}^{kl} := \frac{1}{4} ( c_{kqls} + c_{qkls} + c_{kqsl} + c_{qksl} ).
\end{equation*}

Without invoking the symmetries of the elasticity tensor, we can show the assertion $K_{ij} = K_{ji}$ by using the divergence theorem and Lemmas~\ref{lem1} and \ref{lem3}.
For $k,l = 1, \dots,d$, we compute the $(k,l)$-component of the matrix $K_{ij}-K_{ji}$ as follows:
\begin{align*}
    K_{ij}^{kl} - K_{ji}^{kl} & = \sum_{q,s=1}^{d}\sum_{\alpha = 1}^{\alpha_0} C_{qs}^{kl}\int_{T_\alpha}(\varphi_{j,q}\varphi_{i,s} - \varphi_{j,s}\varphi_{i,q})dx\\ 
    & = \sum_{q,s=1}^{d} C_{qs}^{kl}\sum_{\alpha = 1}^{\alpha_0}\int_{\partial T_\alpha}\varphi_j(\nu_q^\alpha\varphi_{i,s} -\nu_s^\alpha\varphi_{i,q})d\mathrm{S}\\
    & = \sum_{q,s=1}^{d} C_{qs}^{kl} \biggl( \sum_{\alpha = 1}^{\alpha_0}\int_{\partial_0 T_\alpha}\varphi_j(\nu_q^\alpha\varphi_{i,s} -\nu_s^\alpha\varphi_{i,q})d\mathrm{S}\\
    &+ \sum_{\alpha = 1}^{\alpha_0}\int_{\partial_1 T_\alpha}\varphi_{j}(\nu_q^\alpha\varphi_{i,s} -\nu_s^\alpha\varphi_{i,q})d\mathrm{S}\biggr) \begin{array}{ll}
          \text{ ($\partial T_\alpha = \partial_0 T_\alpha \cup \partial_1 T_\alpha$}, \\
          \text{ $ \mathcal{H}^{d-1}(\partial_0 T_\alpha \cap \partial_1 T_\alpha) = 0$)}
    \end{array}\\
    & = \sum_{q,s=1}^{d} C_{qs}^{kl} \sum_{\alpha = 1}^{\alpha_0}\int_{\partial_1 T_\alpha}\varphi_j(\nu_q^\alpha\varphi_{i,s} -\nu_s^\alpha\varphi_{i,q})d\mathrm{S} \text{ ~~(By Lemma \ref{lem3})} \\
    & = \sum_{q,s=1}^{d} C_{qs}^{kl} \sum_{\alpha = 1}^{\alpha_0} \biggl(\int_{\partial_i T_\alpha}\varphi_{j}(\nu_q^\alpha\varphi_{i,s} -\nu_s^\alpha\varphi_{i,q})d\mathrm{S}\\
    &+ \int_{\partial_j T_\alpha}\varphi_{j}(\nu_q^\alpha\varphi_{i,s} -\nu_s^\alpha\varphi_{i,q})d\mathrm{S} \biggr) \begin{array}{ll}
          \text{ ($\partial_1 T_\alpha = \partial_i T_\alpha \cup \partial_j T_\alpha$}, \\
          \text{ $ \mathcal{H}^{d-1}(\partial_i T_\alpha \cap \partial_j T_\alpha) = 0$)}
    \end{array}\\
    & = \sum_{q,s=1}^{d} C_{qs}^{kl} \sum_{\alpha = 1}^{\alpha_0} \int_{\partial_j T_\alpha}\varphi_{j}(\nu_q^\alpha\varphi_{i,s} -\nu_s^\alpha\varphi_{i,q})d\mathrm{S} \text{ ($\varphi_j = 0$ on $\partial_i T_\alpha$)} \\
    & = 0. \text{ ~~(By Lemma \ref{lem1})}
\end{align*}
Since $K_{ij}^T = K_{ji}$, then $K_{ij}^T = K_{ij}$ holds as well.
\end{proof}
\begin{figure}[h]
    \centering
    \includegraphics[width=0.5\textwidth]{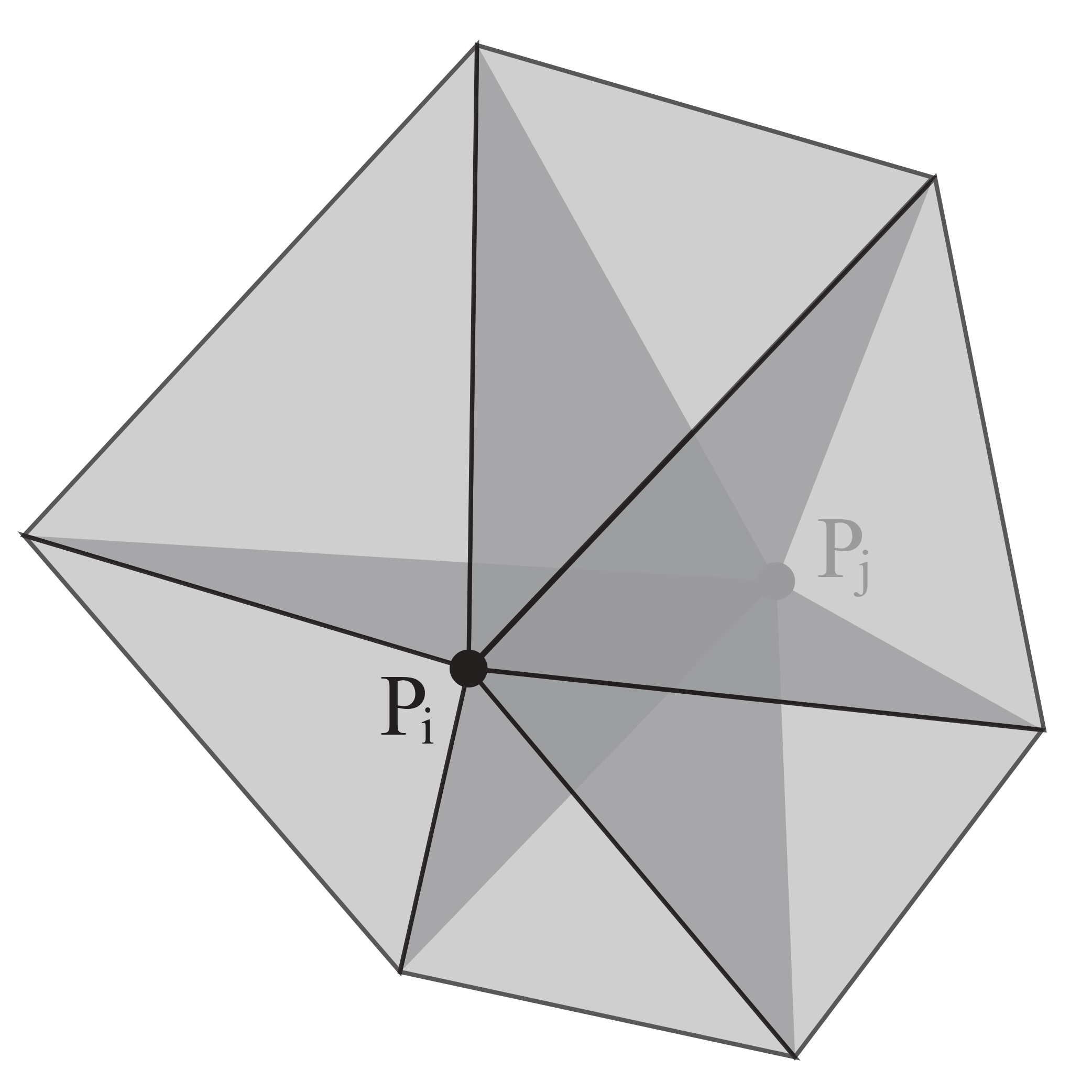}
    \caption{All the tetrahedral elements sharing two nodal points in a 3-dimensional tetrahedral mesh}
    \label{fig4}
\end{figure}

\section{Positive-definiteness of the spring constant}
\label{sec5}
\setcounter{equation}{0}
Under the isotropy assumption, we present a necessary and sufficient condition for the positive-definiteness of the spring constant, along with a sufficient condition in terms of Poisson's ratio and mesh regularity.
In the following, let $\lambda$ and $\mu$ be the standard 3D Lam\'{e} parameters. We note that they coincide with the 2D parameters in the plane strain formulation. We assume that $\mu >0$, $ \lambda + \frac{2}{d}\mu >0$, and that the elasticity tensor satisfies the isotropy condition
\begin{equation}\label{elasiso}
    c_{pqrs} = \lambda\delta_{pq}\delta_{rs} + \mu (\delta_{pr}\delta_{qs} + \delta_{ps}\delta_{qr}).
\end{equation}
Substituting \eqref{elasiso} into \eqref{kijns}, for $k,l = 1, \dots,d$; we get 
\begin{align*}
    K_{ij}^{kl} &= -\sum_{q,s=1}^{d} \int_{\Omega}  \big(\lambda\delta_{kq}\delta_{ls} + \mu (\delta_{kl}\delta_{qs} + \delta_{ks}\delta_{ql})\big) \varphi_{j,s} \varphi_{i,q}  dx\\
    &= - \int_{\Omega} \sum_{q,s=1}^{d}  (\lambda\delta_{kq}\delta_{ls} + \mu \delta_{ks}\delta_{ql}) \varphi_{j,s} \varphi_{i,q}  dx - \mu \int_{\Omega} \sum_{q,s=1}^{d} \delta_{kl}\delta_{qs} \varphi_{j,s} \varphi_{i,q}  dx\\
    &= - \lambda \int_{\Omega} \varphi_{j,l} \varphi_{i,k}   dx - \mu \int_{\Omega} \varphi_{j,k} \varphi_{i,l}  dx - \mu \delta_{kl} \int_{\Omega} \sum_{q=1}^{d} \varphi_{j,q} \varphi_{i,q}  dx.
\end{align*}
Under the assumption \eqref{A-Pij}, we note that by Theorem \ref{th1}, we have 
\begin{align*}
    \int_{\Omega} \varphi_{j,l} \varphi_{i,k}   dx = \int_{\Omega} \varphi_{j,k} \varphi_{i,l}  dx \quad (k,l = 1, \dots, d),
\end{align*}
then $K_{ij}$ can be written as 
\begin{equation*}
    K_{ij} = (\lambda + \mu) A_{ij} + \mu\gamma_{ij} I,
\end{equation*}
where $I$ is the $(d \times d)$ identity matrix, $A_{ij} = (A_{ij}^{kl})_{k,l=1,\cdots,d} \in \mathbb{R}^{d \times d}_{\rm sym}$, $\gamma_{ij}\in \mathbb{R}$, and they are given by
\begin{align*}
    A_{ij}^{kl} = -\int_{\Omega}\varphi_{j,k}\varphi_{i,l} \,dx,\quad
    \gamma_{ij} = -\int_{\Omega} \nabla\varphi_j \cdot \nabla\varphi_i \,dx.
\end{align*}

\begin{Th} \label{th2}    
Under the assumption \eqref{A-Pij}, $K_{ij}$ is positive-definite if and only if
\begin{equation*}
    \max_{\lvert\xi \lvert = 1}\left(\int_{\Omega} (\nabla\varphi_j \cdot \xi) (\nabla\varphi_i \cdot \xi) dx\right) < \frac{-\mu}{\lambda + \mu} \int_{\Omega} \nabla\varphi_j \cdot \nabla\varphi_i dx
\end{equation*}
holds.
\end{Th}

\begin{proof}
We first remark that $\lambda +\mu >0$ holds under the assumption on the Lam\'e parameters.
Let $\eta_1 \geq \cdots \geq \eta_d$ be the eigenvalues of $A_{ij}$, then the eigenvalues of $K_{ij}$ denoted $\zeta_1 \geq \cdots \geq \zeta_d$ will be given by
\begin{equation*}
    \zeta_l = (\lambda + \mu) \eta_l + \gamma_{ij}\mu, \qquad (l = 1 \cdots d).
\end{equation*}
It is easy to see that $K_{ij}$ is positive-definite if and only if $-\eta_d < \frac{\gamma_{ij} \mu}{\lambda + \mu}$, where $\eta_d$ is the smallest eigenvalue of $A_{ij}$ given by
\begin{equation*}
    \eta_d = \min_{\lvert \xi \rvert= 1, \xi \in \mathbb{R}^d }\big((A_{ij}\xi) \cdot \xi \big).
\end{equation*}
We have
\begin{align*}
        (A_{ij}\xi)\cdot \xi = -\int_{\Omega} (\nabla\varphi_j \cdot \xi) (\nabla\varphi_i \cdot \xi)\, dx,
\end{align*}
which yields
\begin{equation*}
    -\eta_d = \max_{\lvert \xi \lvert = 1}\left(\int_{\Omega} (\nabla\varphi_j \cdot \xi) (\nabla\varphi_i \cdot \xi) \,dx \right).
\end{equation*}
\end{proof}

Let $\nu:= \frac{\lambda}{2(\lambda+\mu)}$ be the standard 3D Poisson's ratio of the material occupying the domain $\Omega$. This is equivalent to Poisson's ratio used in the 2D plane strain formulation, and we note that $-1<\nu<0.5$. Let $\theta^\alpha$ be the angle (or dihedral angle) between $\partial_j T_\alpha$ and $\partial_i T_\alpha$.
For $l \in \{ i ,j \}$, we denote $\nabla \varphi_l^\alpha := \nabla \varphi_l \vert_{T_\alpha} $. We note that $\nabla \varphi_l^\alpha = - \lvert \nabla \varphi_l^\alpha  \lvert  n_l^\alpha$, where $n_l^\alpha$ is the unit outward normal vector on $\partial_l T_\alpha$. 
\begin{Th}\label{th3}
Under the assumption \eqref{A-Pij}, $K_{ij}$ is positive-definite, if for all $\alpha \in \{1\dots \alpha_0\}$,
\begin{equation*}
    \cos{\theta^\alpha} > \frac{1}{3-4\nu}
\end{equation*}
holds.    
\end{Th}

\begin{proof}
According to Theorem~\ref{th2}, and the following inequality
\begin{align*}
        \max_{\lvert  \xi \lvert  = 1} \big(\int_{\Omega} (\nabla\varphi_j \cdot \xi) (\nabla\varphi_i \cdot \xi) dx \big) &= \max_{\lvert  \xi \lvert  = 1}\big(\sum_{\alpha=1}^{\alpha_0} \int_{T_\alpha} (\nabla\varphi_j \cdot \xi) (\nabla\varphi_i \cdot \xi) dx \big)\\ 
        &\leq \sum_{\alpha=1}^{\alpha_0} \max_{\lvert  \xi \lvert  = 1}\big( \int_{T_\alpha} (\nabla\varphi_j \cdot \xi) (\nabla\varphi_i \cdot \xi) dx \big)\\
        &= \sum_{\alpha=1}^{\alpha_0} \lvert T_\alpha \rvert \,\lvert  \nabla \varphi_j^\alpha  \lvert  \,\lvert  \nabla \varphi_i^\alpha \lvert  \,\max_{\lvert  \xi \lvert  = 1}\big( (n_j^\alpha \cdot \xi)(n_i^\alpha \cdot \xi) \big), 
\end{align*}
the positive-definiteness of $K_{ij}$ is guaranteed if
\begin{align*}\label{eqpg}
    \sum_{\alpha=1}^{\alpha_0} \lvert T_\alpha \rvert \, \lvert  \nabla \varphi_j^\alpha  \lvert \,  \lvert  \nabla \varphi_i^\alpha \lvert  \max_{\lvert  \xi \lvert  = 1}\big( (n_j^\alpha \cdot \xi)(n_i^\alpha \cdot \xi) \big) &< \frac{-\mu}{\lambda + \mu}\int_{\Omega} \nabla\varphi_j \cdot \nabla\varphi_i dx \\
    &= \sum_{\alpha = 1}^{\alpha_0} \lvert T_\alpha \rvert \,\lvert  \nabla \varphi_j^\alpha  \lvert \, \lvert  \nabla \varphi_i^\alpha \lvert  \frac{-(n_j^\alpha \cdot n_i^\alpha) \mu}{\lambda + \mu}
\end{align*}
holds. Consequently, with Lemma \ref{lem4} it is clear that if  
\begin{equation*}
    \frac{(1+ n_j^\alpha \cdot n_i^\alpha)}{2} < \frac{-(n_j^\alpha \cdot n_i^\alpha) \mu}{\lambda + \mu}
\end{equation*}
holds for all $\alpha \in \{1\dots \alpha_0\}$, then $K_{ij}$ is positive-definite. Finally, we make use of the following relations
\begin{align*}
    n_j^\alpha \cdot n_i^\alpha = -\cos{\theta^\alpha},\quad
    \frac{\mu}{\lambda + \mu} = 1 - 2\nu. 
\end{align*}
\end{proof}

\begin{Rem}
\label{remequi}
When $d = 2$, and $T_\alpha$ is an equilateral triangle ($\theta^\alpha = \pi / 3$), for $\alpha = 1,2$. Then by Lemma \ref{lem4} we have 
\begin{align*}
        \max_{\lvert  \xi \lvert  = 1} \big(\int_{\Omega} (\nabla\varphi_j \cdot \xi) (\nabla\varphi_i \cdot \xi) dx \big) &=  \sum_{\alpha=1}^{\alpha_0} \max_{\lvert  \xi \lvert  = 1}\big( \int_{T_\alpha} (\nabla\varphi_j \cdot \xi) (\nabla\varphi_i \cdot \xi) dx \big),
\end{align*}
and it follows that in the plane strain formulation, $K_{ij}$ is positive-definite if and only if $\nu<1/4 $. This coincides with the condition given in \cite{NK14}. 
For plane stress, we have $\nu_{2D} = \frac{\nu}{1-\nu}$ \cite{mhs20}, and it follows that $K_{ij}$ is positive-definite if and only if $\nu_{2D}<1/3 $.

This case is especially relevant because if we look for a triangulation of the 2D space such that the maximum of all triangular angles is minimized, then the minimum will be $\pi / 3$ achieved by a regular triangular tessellation.  
    
\end{Rem}

\begin{Rem}
\label{rem3d}
When $d = 3$, any triangulation must contain some dihedral angles greater than or equal to $72^\circ$ \cite{EPS04}. Thus, if we wish for a spring-block system with only positive-definite springs, then our condition only works for materials with $\nu <-0.0591$. To the best of our knowledge, the closest known triangulation to this upper limit is given in \cite{EPS04}, with a maximum dihedral angle of $74.20^\circ$, and thus $\nu <-0.1682$. When $\nu = 0$, we get $\theta^\alpha < \arccos(1/3)$, which is the dihedral angle of a regular tetrahedron. 
\end{Rem}

\section{Numerical examples}\label{sec6} 
\setcounter{equation}{0}
While we have shown that the symmetry of the spring constant is always true by Theorem~\ref{th1}, the positivity is not. Theorem \ref{th2} is a sufficient and necessary condition for the positivity of a spring constant, but it is difficult to work with when building a mesh; for this reason, we introduced in Theorem \ref{th3} a sufficient condition in terms of Poisson's ratio and the local geometry of the mesh. Constructing a mesh that satisfies this condition everywhere is difficult, except in two dimensions and for small values of Poisson's ratio, as in Remark \ref{remequi}. However, since this condition is only sufficient, it is possible for a spring constant not to satisfy it and be positive-definite. To investigate the relevance of this possibility, we perform a numerical evaluation of the spring constants derived from several sample meshes for varying values of Poisson's ratio. 

\begin{figure}[ht]
    \centering
    \subfloat[]{\label{fig5a}\includegraphics[width=.22\linewidth]{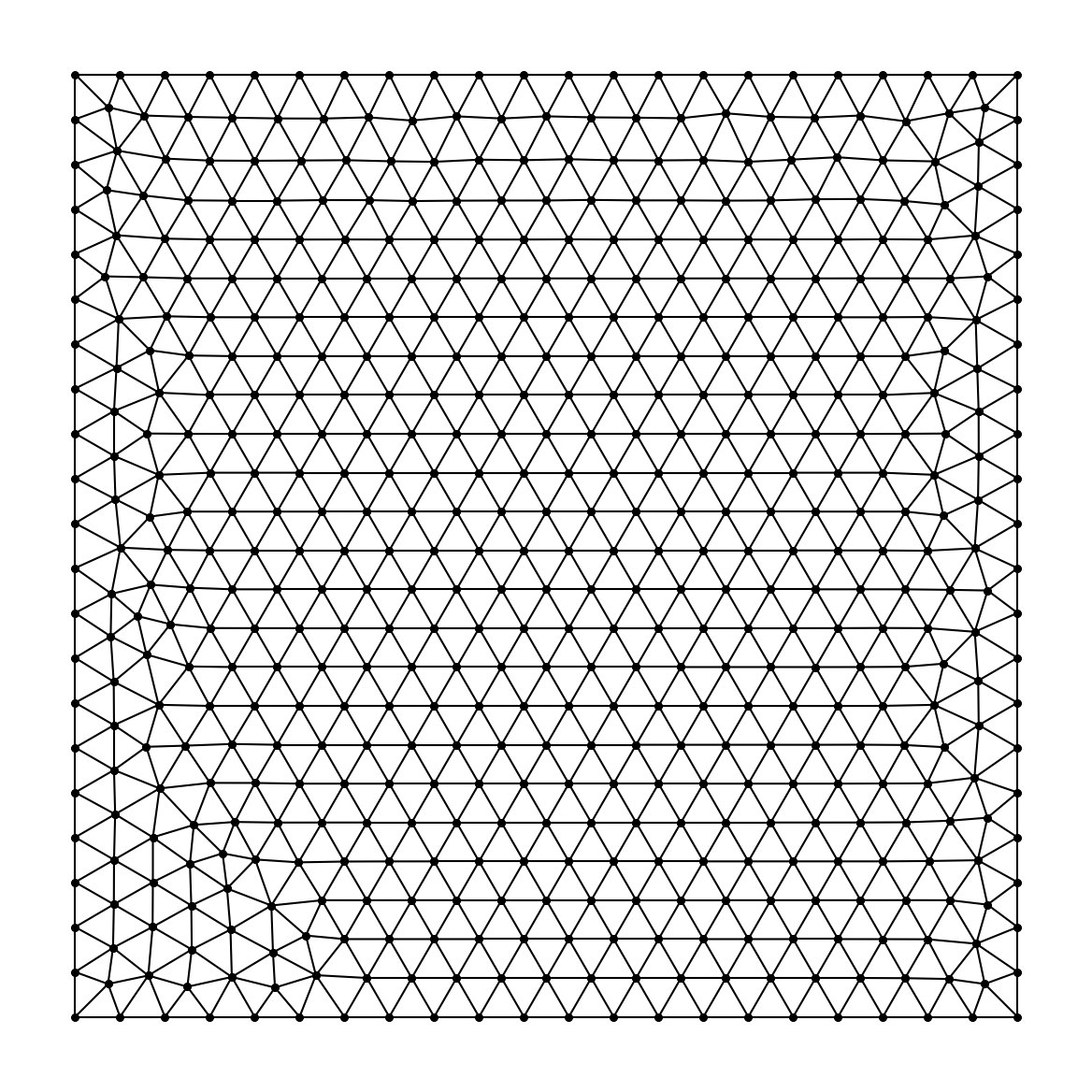}}\quad
    \subfloat[]{\label{fig5b}\includegraphics[width=.22\linewidth]{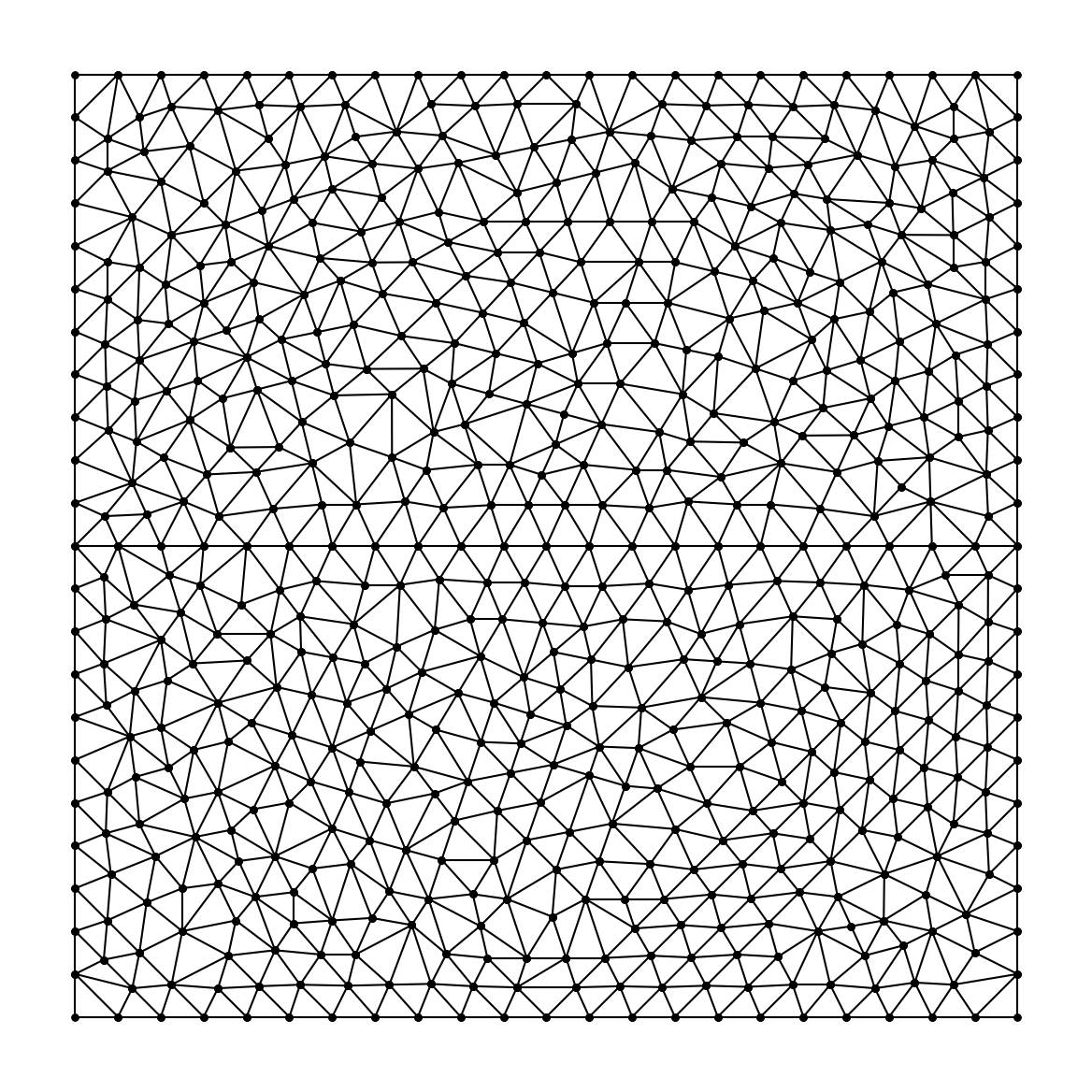}}\quad
    \subfloat[]{\label{fig5c}\includegraphics[width=.22\linewidth]{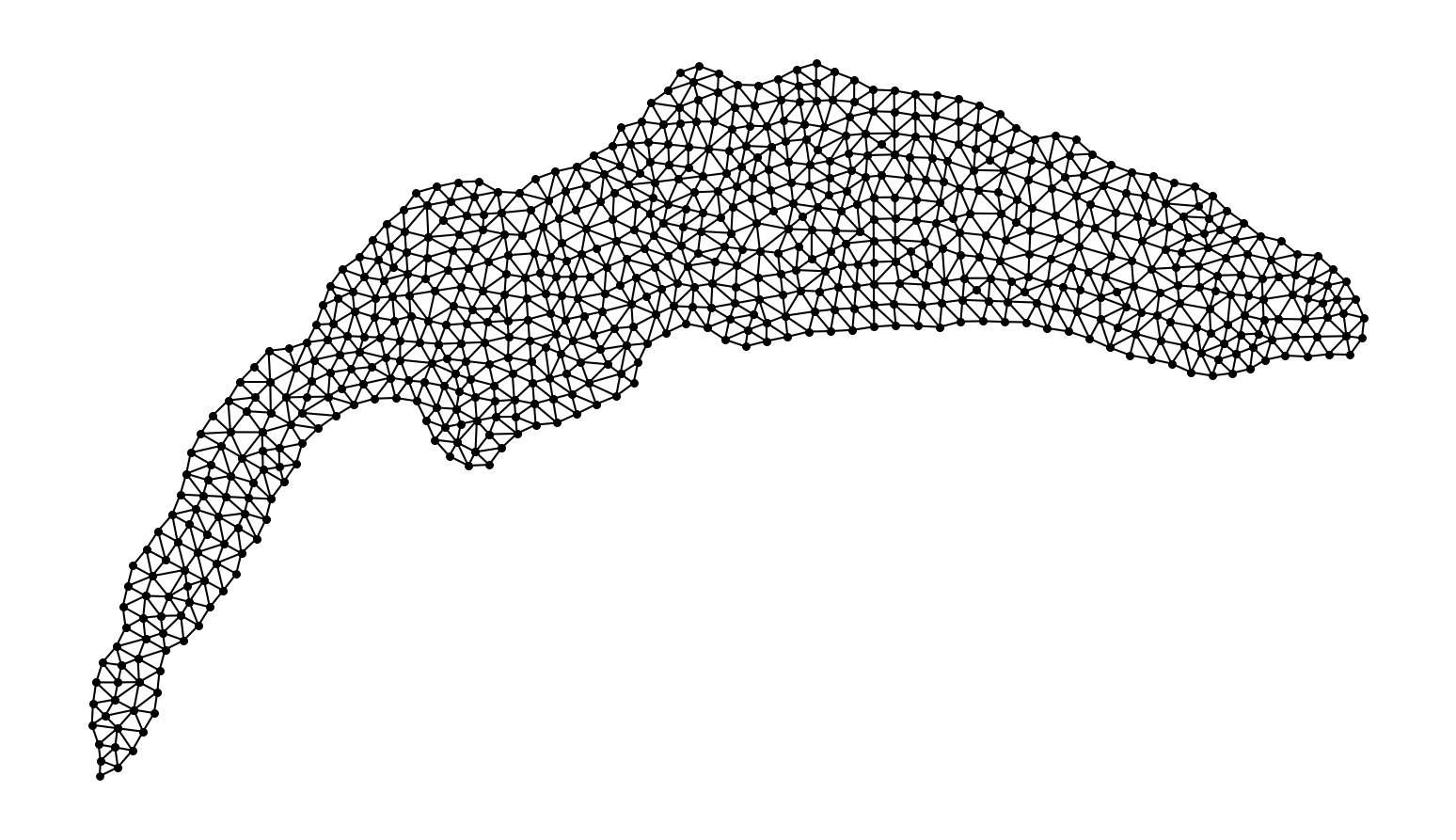}}\quad
    \subfloat[]{\label{fig5d}\includegraphics[width=.22\linewidth]{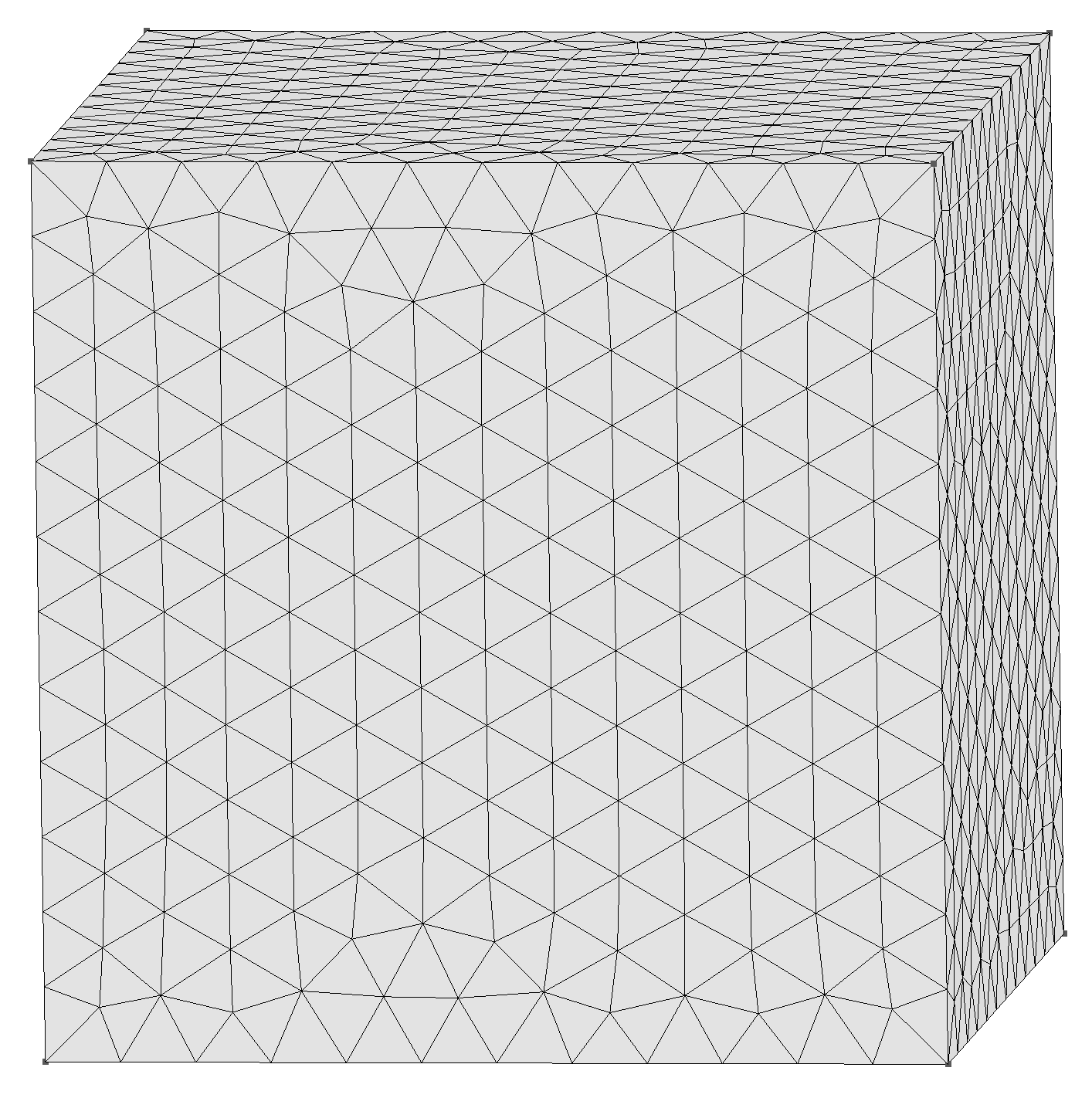}}\par
    \caption{Sample FEM meshes}
    \label{fig5}
\end{figure}

We consider three 2-dimensional meshes with triangular elements; one is constructed on an irregular shape domain, and the other two are on square domains, with one of them being more equilateral (most of its elements are equilateral triangles). We also consider a three-dimensional mesh with tetrahedral elements on a cubical domain (Figure~\ref{fig5}).

\begin{figure}[h]
    \centering
    \includegraphics[width=0.55\textwidth]{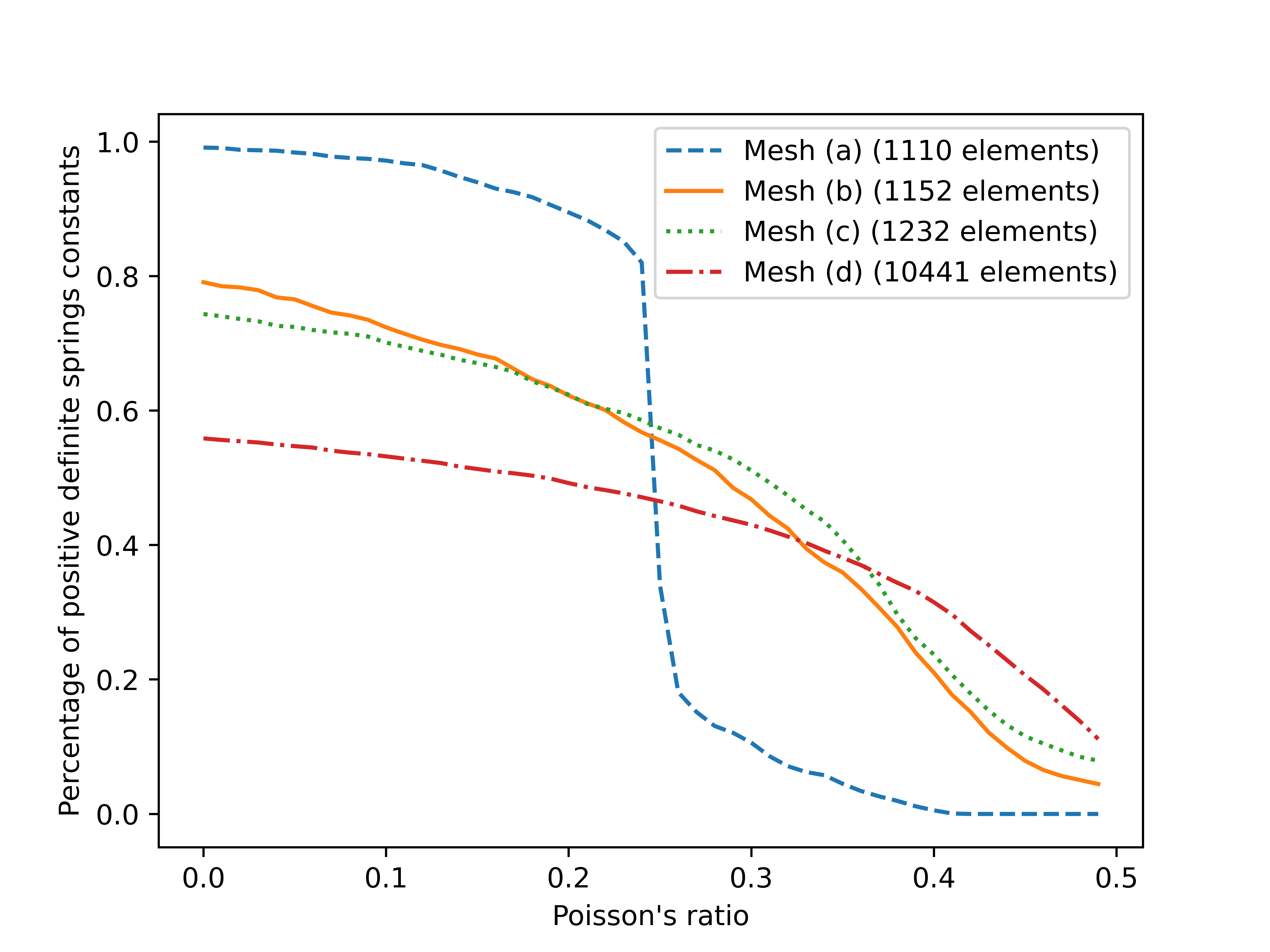}
    \caption{The percentage of the positive-definite spring constants derived from some FEM meshes in terms of Poisson's ratio}
    \label{fig6}
\end{figure}

To investigate the global effect of the Poisson's ratio, assuming that the material occupying the domain is isotropic, we derive a spring-block system from the P1-FEM on every sample mesh as discussed above, and we plot the percentage of the spring constants that are positive-definite as a function of the Poisson's ratio (Figure~\ref{fig6}). The number of springs with positive-definite constants decreases as Poisson's ratio increases, regardless of the mesh. In the more equilateral two-dimensional mesh, almost all of the spring constants are positive-definite for small values of the Poisson's ratio; we also observe a sharp drop around the value $\nu = 0.25$ in accordance with Remark \ref{remequi}. In three dimensions, even for small values of Poisson's ratio, a large portion of the derived spring constants are not positive-definite (Remark \ref{rem3d}). 

Additionally, we plot a color map on a square mesh to showcase the effect of the local geometry of the mesh on the smallest eigenvalue of the derived spring constant. The virtual springs considered in the P1-FEM derived spring-block system are each associated with a segment line in the mesh. Here, the color around a segment indicates the smallest eigenvalue of the associated spring constant (Figure~\ref{fig7}). 
\begin{figure}[ht]
    \centering
    \subfloat[]{\label{fig1a}\includegraphics[width=.48\linewidth]{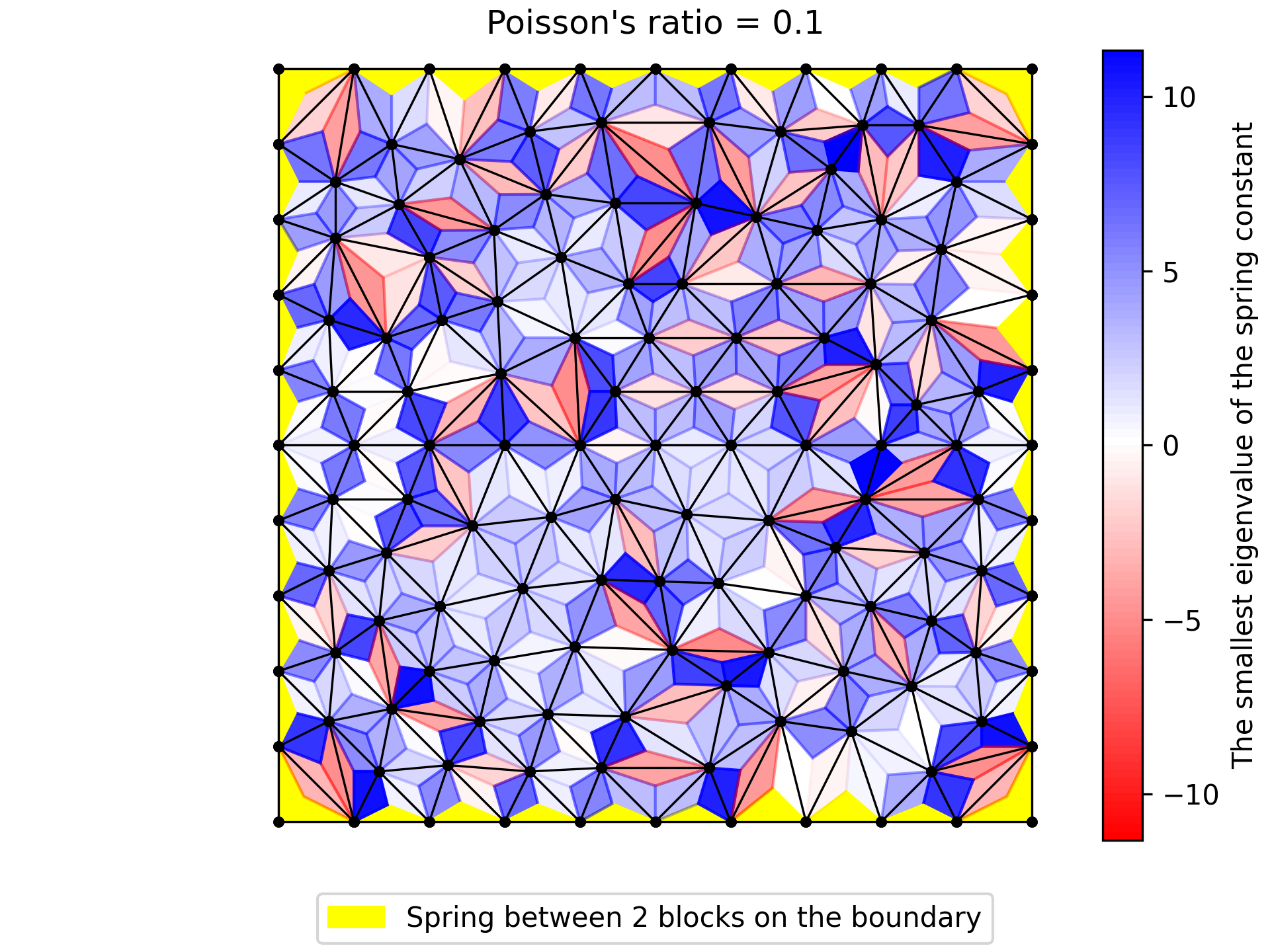}}\quad
    \subfloat[]{\label{fig1a}\includegraphics[width=.48\linewidth]{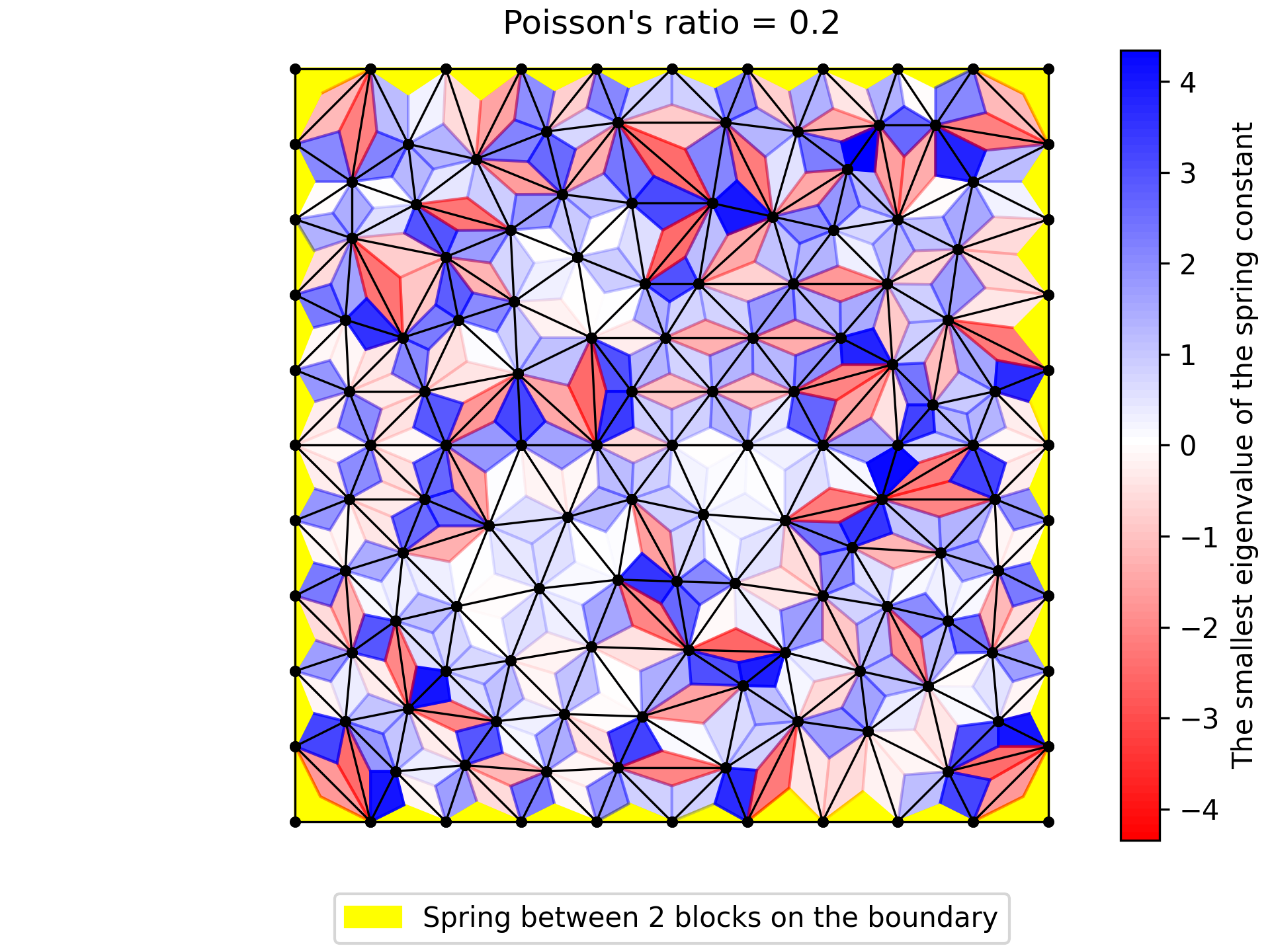}}\quad
    \subfloat[]{\label{fig1a}\includegraphics[width=.48\linewidth]{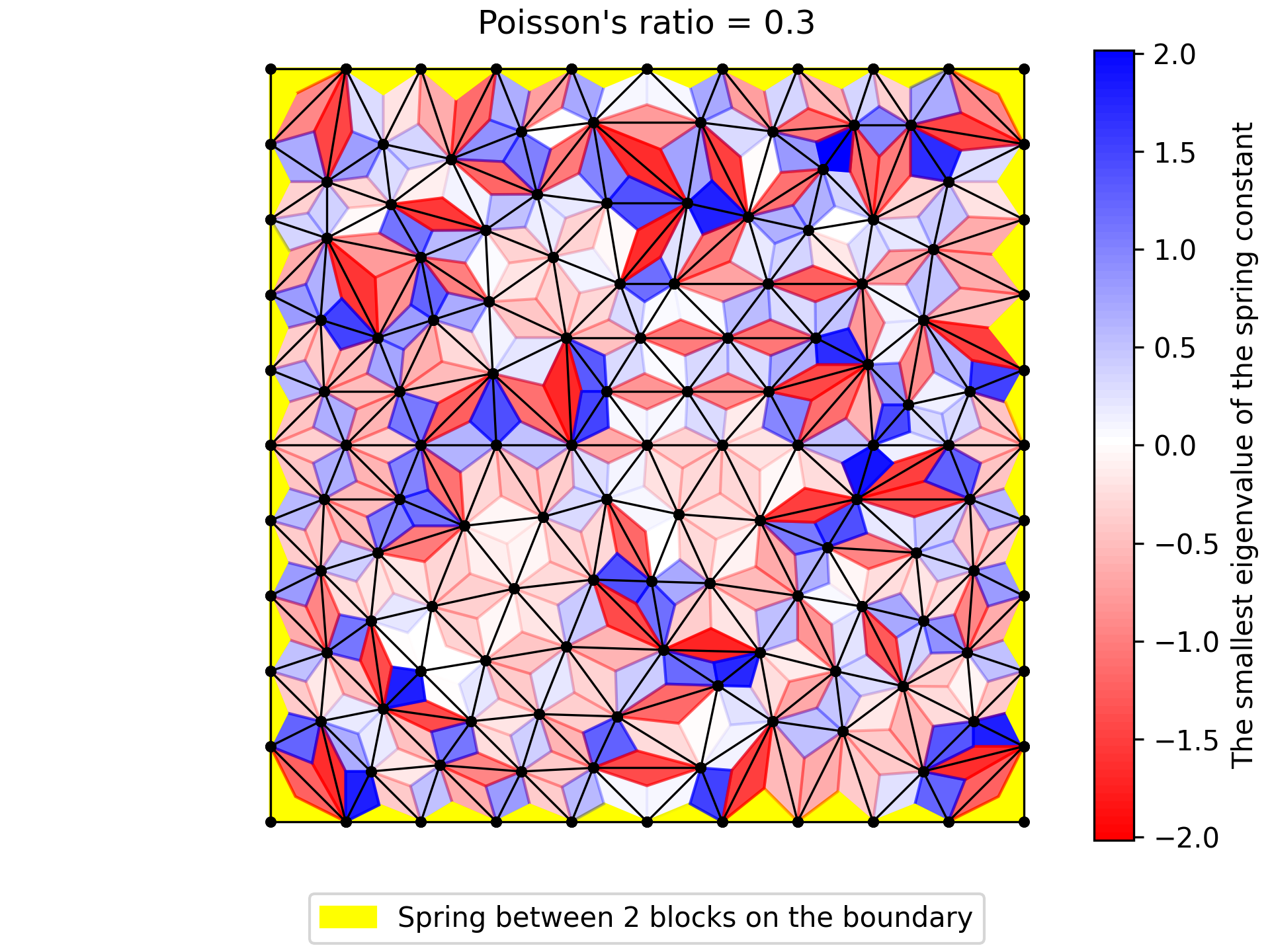}}\quad
    \subfloat[]{\label{fig1b}\includegraphics[width=.48\linewidth]{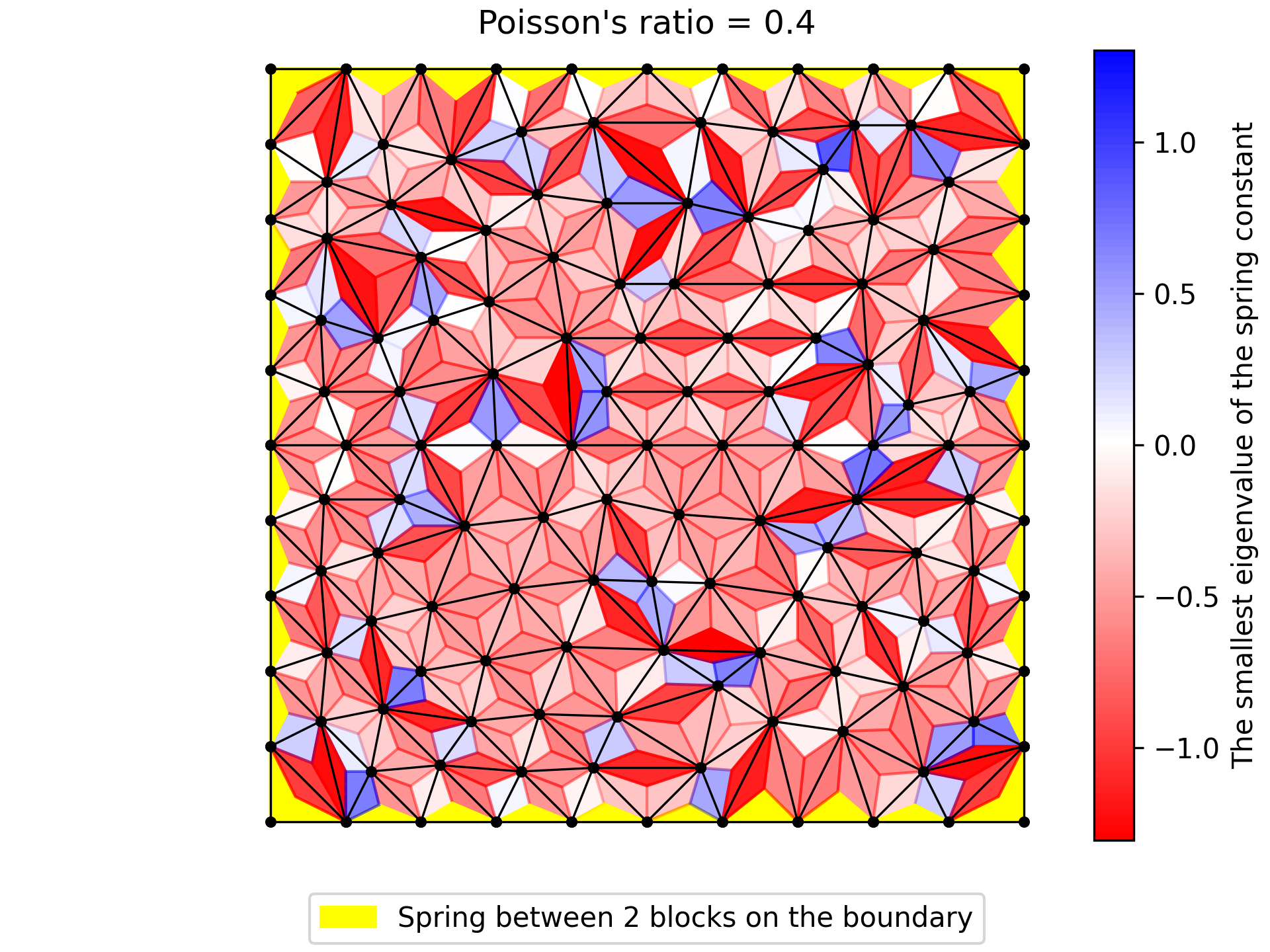}}\par
    \caption{The smallest eigenvalue of the spring constant associated with a segment line}
    \label{fig7}
\end{figure}
As expected, the segments encompassed by smaller angles, as discussed in the previous section, produce positive-definite spring constants for a larger interval of Poisson's ratio. 

Indeed, in both cases, the numerical observations agree with our proposed condition for positivity by analytical discussions. While we can control the positive-definiteness of a given spring constant by altering the mesh locally, there is a limit on the global behavior due to the problem of space tiling. These results suggest that relaxing the proposed condition on the mesh regularity in any meaningful way may not be possible.

\section{Conclusion}\label{sec7} 
\setcounter{equation}{0}
In this study, we have further investigated the spring constants derived from P1-FEM. To ensure the solvability of this model when employed in fracture analysis, the spring constants are sought to be symmetric and positive-definite as studied in \cite{KN13}. We have shown that the spring constants are symmetric for any homogeneous anisotropic elasticity tensor in a unified approach in both 2D and 3D. For the positivity of the spring constants, while it is not always true, we could guarantee it for isotropic elasticity by imposing a sufficient condition in terms of the FEM mesh regularity and Poisson's ratio. This condition is easy to impose locally. However, the same can not be said if we wish for it to hold everywhere in the mesh. On that front, we can only guarantee it for small values of the Poisson's ratio in 2 dimensions. This remark is further supported by the numerical results we presented. 

Although our final findings are pessimistic about deriving the perfect spring-block model consistent with linear elasticity and allowing mathematical solvability of the fracture models in 2D and 3D, the mathematical analysis we carried out will be crucial in future works.

\funding{This work was partially supported by JSPS KAKENHI Grant Numbers JP24H00184, JP25K00920, JP20H01823, JP21H04431, JP24H00188 and JP25K00920; JST CREST Grant Number JPMJCR2014; and MEXT Scholarship (No. 230707).}

\dataava{Data will be made available on reasonable request.}

\complia{On behalf of all authors, the corresponding author states that there is no conflict of interest.}

\ethical{Not applicable.}

\bibliographystyle{camc}
\bibliography{camc}


\begin{thebibliography}{17}
\ifx \bisbn   \undefined \def \bisbn  #1{ISBN #1}\fi
\ifx \binits  \undefined \def \binits#1{#1}\fi
\ifx \bauthor  \undefined \def \bauthor#1{#1}\fi
\ifx \batitle  \undefined \def \batitle#1{#1}\fi
\ifx \bjtitle  \undefined \def \bjtitle#1{#1}\fi
\ifx \bvolume  \undefined \def \bvolume#1{\textbf{#1}}\fi
\ifx \byear  \undefined \def \byear#1{#1}\fi
\ifx \bissue  \undefined \def \bissue#1{#1}\fi
\ifx \bfpage  \undefined \def \bfpage#1{#1}\fi
\ifx \blpage  \undefined \def \blpage #1{#1}\fi
\ifx \burl  \undefined \def \burl#1{\textsf{#1}}\fi
\ifx \doiurl  \undefined \def \doiurl#1{\url{https://doi.org/#1}}\fi
\ifx \betal  \undefined \def \betal{\textit{et al.}}\fi
\ifx \binstitute  \undefined \def \binstitute#1{#1}\fi
\ifx \binstitutionaled  \undefined \def \binstitutionaled#1{#1}\fi
\ifx \bctitle  \undefined \def \bctitle#1{#1}\fi
\ifx \beditor  \undefined \def \beditor#1{#1}\fi
\ifx \bpublisher  \undefined \def \bpublisher#1{#1}\fi
\ifx \bbtitle  \undefined \def \bbtitle#1{#1}\fi
\ifx \bedition  \undefined \def \bedition#1{#1}\fi
\ifx \bseriesno  \undefined \def \bseriesno#1{#1}\fi
\ifx \blocation  \undefined \def \blocation#1{#1}\fi
\ifx \bsertitle  \undefined \def \bsertitle#1{#1}\fi
\ifx \bsnm \undefined \def \bsnm#1{#1}\fi
\ifx \bsuffix \undefined \def \bsuffix#1{#1}\fi
\ifx \bparticle \undefined \def \bparticle#1{#1}\fi
\ifx \barticle \undefined \def \barticle#1{#1}\fi
\bibcommenthead
\ifx \bconfdate \undefined \def \bconfdate #1{#1}\fi
\ifx \botherref \undefined \def \botherref #1{#1}\fi
\ifx \url \undefined \def \url#1{\textsf{#1}}\fi
\ifx \bchapter \undefined \def \bchapter#1{#1}\fi
\ifx \bbook \undefined \def \bbook#1{#1}\fi
\ifx \bcomment \undefined \def \bcomment#1{#1}\fi
\ifx \oauthor \undefined \def \oauthor#1{#1}\fi
\ifx \citeauthoryear \undefined \def \citeauthoryear#1{#1}\fi
\ifx \endbibitem  \undefined \def \endbibitem {}\fi
\ifx \bconflocation  \undefined \def \bconflocation#1{#1}\fi
\ifx \arxivurl  \undefined \def \arxivurl#1{\textsf{#1}}\fi
\csname PreBibitemsHook\endcsname

\bibitem[\protect\citeauthoryear{Bolander et~al.}{2021}]{BE21}
\begin{barticle}
\bauthor{\bsnm{Bolander}, \binits{J.E.}},
\bauthor{\bsnm{Eliáš}, \binits{J.}},
\bauthor{\bsnm{Cusatis}, \binits{G.}},
\bauthor{\bsnm{Nagai}, \binits{K.}}:
\batitle{{Discrete mechanical models of concrete fracture}}.
\bjtitle{Engineering Fracture Mechanics}
\bvolume{275},
\bfpage{108030}
(\byear{2021}).
\doiurl{10.1016/j.engfracmech.2021.108030}
\end{barticle}
\endbibitem

\bibitem[\protect\citeauthoryear{Brenner and Scott}{2008}]{FEM2}
\begin{bbook}
\bauthor{\bsnm{Brenner}, \binits{S.C.}},
\bauthor{\bsnm{Scott}, \binits{L.R.}}:
\bbtitle{{The Mathematical Theory of Finite Element Methods}}.
\bsertitle{Texts in Applied Mathematics},
vol. \bseriesno{15}.
\bpublisher{Springer},
\blocation{USA}
(\byear{2008}).
\doiurl{10.1007/978-0-387-75934-0}
\end{bbook}
\endbibitem

\bibitem[\protect\citeauthoryear{Cervera et~al.}{2022}]{CB22}
\begin{barticle}
\bauthor{\bsnm{Cervera}, \binits{M.}},
\bauthor{\bsnm{Barbat}, \binits{G.B.}},
\bauthor{\bsnm{Chiumenti}, \binits{M.}},
\bauthor{\bsnm{Wu}, \binits{J.-Y.}}:
\batitle{{A Comparative Review of XFEM, Mixed FEM and Phase‐Field Models for Quasi‐brittle Cracking}}.
\bjtitle{Archives of Computational Methods in Engineering}
\bvolume{29},
\bfpage{1009}--\blpage{1083}
(\byear{2022}).
\doiurl{10.1007/s11831-021-09604-8}
\end{barticle}
\endbibitem

\bibitem[\protect\citeauthoryear{Chen et~al.}{2014}]{CHEN14}
\begin{barticle}
\bauthor{\bsnm{Chen}, \binits{H.}},
\bauthor{\bsnm{Lin}, \binits{E.}},
\bauthor{\bsnm{Liu}, \binits{Y.}}:
\batitle{A novel volume-compensated particle method for 2d elasticity and plasticity analysis}.
\bjtitle{International Journal of Solids and Structures}
\bvolume{51}(\bissue{9}),
\bfpage{1819}--\blpage{1833}
(\byear{2014}).
\doiurl{10.1016/j.ijsolstr.2014.01.025}
\end{barticle}
\endbibitem

\bibitem[\protect\citeauthoryear{Ciarlet}{2002}]{FEM1}
\begin{bbook}
\bauthor{\bsnm{Ciarlet}, \binits{P.G.}}:
\bbtitle{{The Finite Element Method for Elliptic Problems}}.
\bsertitle{Classics in Applied Mathematics}.
\bpublisher{Society for Industrial and Applied Mathematics},
\blocation{USA}
(\byear{2002}).
\doiurl{10.1137/1.9780898719208}
\end{bbook}
\endbibitem

\bibitem[\protect\citeauthoryear{Domaneschi et~al.}{2020}]{MC20}
\begin{barticle}
\bauthor{\bsnm{Domaneschi}, \binits{M.}},
\bauthor{\bsnm{Pellecchia}, \binits{C.}},
\bauthor{\bsnm{De~Iuliis}, \binits{E.}},
\bauthor{\bsnm{Cimellaro}, \binits{G.P.}},
\bauthor{\bsnm{Morgese}, \binits{M.}},
\bauthor{\bsnm{Khalil}, \binits{A.A.}},
\bauthor{\bsnm{Ansari}, \binits{F.}}:
\batitle{{Collapse analysis of the Polcevera viaduct by the applied element method}}.
\bjtitle{Engineering Structures}
\bvolume{214},
\bfpage{110659}
(\byear{2020}).
\doiurl{10.1016/j.engstruct.2020.110659}
\end{barticle}
\endbibitem

\bibitem[\protect\citeauthoryear{Duvaut and Lions}{1976}]{D-L1976}
\begin{bbook}
\bauthor{\bsnm{Duvaut}, \binits{G.}},
\bauthor{\bsnm{Lions}, \binits{J.-L.}}:
\bbtitle{Inequalities in Mechanics and Physics},
\bedition{1}st edn.
\bsertitle{Grundlehren der mathematischen Wissenschaften},
vol. \bseriesno{219},
p. \bfpage{400}.
\bpublisher{Springer},
\blocation{Germany}
(\byear{1976}).
\doiurl{10.1007/978-3-642-66165-5}
\end{bbook}
\endbibitem

\bibitem[\protect\citeauthoryear{Eppstein et~al.}{2004}]{EPS04}
\begin{barticle}
\bauthor{\bsnm{Eppstein}, \binits{D.}},
\bauthor{\bsnm{Sullivan}, \binits{J.M.}},
\bauthor{\bsnm{\"{U}ng\"{o}r}, \binits{A.}}:
\batitle{Tiling space and slabs with acute tetrahedra}.
\bjtitle{Comput. Geom. Theory Appl.}
\bvolume{27}(\bissue{3}),
\bfpage{237}--\blpage{255}
(\byear{2004}).
\doiurl{10.1016/j.comgeo.2003.11.003}
\end{barticle}
\endbibitem

\bibitem[\protect\citeauthoryear{Grunwald et~al.}{2018}]{CA18}
\begin{barticle}
\bauthor{\bsnm{Grunwald}, \binits{C.}},
\bauthor{\bsnm{Khalil}, \binits{A.A.}},
\bauthor{\bsnm{Schaufelberger}, \binits{B.}},
\bauthor{\bsnm{Ricciardi}, \binits{E.M.}},
\bauthor{\bsnm{Pellecchia}, \binits{C.}},
\bauthor{\bsnm{{De Iuliis}}, \binits{E.}},
\bauthor{\bsnm{Riedel}, \binits{W.}}:
\batitle{{Reliability of collapse simulation – Comparing finite and applied element method at different levels}}.
\bjtitle{Engineering Structures}
\bvolume{176},
\bfpage{265}--\blpage{287}
(\byear{2018}).
\doiurl{10.1016/j.engstruct.2018.08.068}
\end{barticle}
\endbibitem

\bibitem[\protect\citeauthoryear{Hori et~al.}{2005}]{MK05}
\begin{barticle}
\bauthor{\bsnm{Hori}, \binits{M.}},
\bauthor{\bsnm{Oguni}, \binits{K.}},
\bauthor{\bsnm{Sakaguchi}, \binits{H.}}:
\batitle{{Proposal of FEM implemented with particle discretization for analysis of failure phenomena}}.
\bjtitle{Journal of the Mechanics and Physics of solids}
\bvolume{53},
\bfpage{681}--\blpage{703}
(\byear{2005}).
\doiurl{10.1016/j.jmps.2004.08.005}
\end{barticle}
\endbibitem

\bibitem[\protect\citeauthoryear{Kimura and Notsu}{2013}]{KN13}
\begin{barticle}
\bauthor{\bsnm{Kimura}, \binits{M.}},
\bauthor{\bsnm{Notsu}, \binits{H.}}:
\batitle{{A mathematical model of fracture phenomena on a spring-block system}}.
\bjtitle{Kyoto University RIMS Kokyuroku}
\bvolume{1848},
\bfpage{171}--\blpage{186}
(\byear{2013})
\end{barticle}
\endbibitem

\bibitem[\protect\citeauthoryear{Kwok et~al.}{2020}]{KK20}
\begin{barticle}
\bauthor{\bsnm{Kwok}, \binits{C.-Y.}},
\bauthor{\bsnm{Duan}, \binits{K.}},
\bauthor{\bsnm{Pierce}, \binits{M.}}:
\batitle{{Modeling hydraulic fracturing in jointed shale formation with the use of fully coupled discrete element method}}.
\bjtitle{Acta Geotechnica}
\bvolume{15},
\bfpage{245}--\blpage{264}
(\byear{2020}).
\doiurl{10.1007/s11440-019-00858-y}
\end{barticle}
\endbibitem

\bibitem[\protect\citeauthoryear{Liu et~al.}{2020}]{LX20}
\begin{barticle}
\bauthor{\bsnm{Liu}, \binits{G.-Y.}},
\bauthor{\bsnm{Xu}, \binits{W.-J.}},
\bauthor{\bsnm{Govender}, \binits{N.}},
\bauthor{\bsnm{Wilke}, \binits{D.N.}}:
\batitle{{A cohesive fracture model for discrete element method based on polyhedral blocks}}.
\bjtitle{Powder Technology}
\bvolume{359},
\bfpage{190}--\blpage{204}
(\byear{2020}).
\doiurl{10.1016/j.powtec.2019.09.068}
\end{barticle}
\endbibitem

\bibitem[\protect\citeauthoryear{Notsu and Kimura}{2014}]{NK14}
\begin{barticle}
\bauthor{\bsnm{Notsu}, \binits{H.}},
\bauthor{\bsnm{Kimura}, \binits{M.}}:
\batitle{{Symmetry and positive definiteness of the tensor-valued spring constant derived from P1-FEM for the equations of linear elasticity}}.
\bjtitle{Networks and Heterogeneous Media}
\bvolume{9}(\bissue{4}),
\bfpage{617}--\blpage{634}
(\byear{2014}).
\doiurl{10.3934/nhm.2014.9.617}
\end{barticle}
\endbibitem

\bibitem[\protect\citeauthoryear{Quaranta et~al.}{2020}]{LL20}
\begin{barticle}
\bauthor{\bsnm{Quaranta}, \binits{L.}},
\bauthor{\bsnm{Maddegedara}, \binits{L.}},
\bauthor{\bsnm{Okinaka}, \binits{T.}},
\bauthor{\bsnm{Hori}, \binits{M.}}:
\batitle{{Application of PDS–FEM to simulate dynamic crack propagation and supershear rupture}}.
\bjtitle{Computational Mechanics}
\bvolume{65},
\bfpage{1289}--\blpage{1304}
(\byear{2020}).
\doiurl{10.1007/s00466-020-01819-z}
\end{barticle}
\endbibitem

\bibitem[\protect\citeauthoryear{Sadd}{2020}]{mhs20}
\begin{bbook}
\bauthor{\bsnm{Sadd}, \binits{M.H.}}:
\bbtitle{Elasticity Theory, Applications, and Numerics},
\bedition{4}th edn.
\bpublisher{Elsevier},
\blocation{United Kingdom}
(\byear{2020}).
\doiurl{10.1016/C2017-0-03720-5}
\end{bbook}
\endbibitem

\end{thebibliography}

\begin{appendices}

\section{}\label{secA1}
\begin{Lem}\label{lem4}
Let $d \in \mathbb{N}$, $d \geq 2$. For $a$, $b$ $\in \mathbb{R}^d$ with $\lvert a \rvert= \lvert b \rvert= 1$, it holds that
\begin{equation*}
    \max_{\lvert x \rvert= 1, x \in \mathbb{R}^d }\big((a \cdot x) (b \cdot x) \big) =  \frac{1+ a \cdot b}{2}.
\end{equation*}
The maximum is achieved by 
\begin{align*}
    \begin{cases}
    {\displaystyle x = \pm \frac{a+b}{\lvert a+b \rvert}} & \mbox{if } a+b \neq 0,\\[5pt]
    x\in \langle a \rangle^\perp & \mbox{if } a+b = 0.
    \end{cases}
\end{align*}
\end{Lem}
\begin{proof}
In the following , let $a,b,x\in \mathbb{R}^d$, $\lvert a \rvert = \lvert b \rvert = \lvert x \rvert = 1$.
1) If $a+b = 0$, then
\begin{align*}
    \max_{\lvert x \rvert= 1, x \in \mathbb{R}^d }((a \cdot x) (b \cdot x) ) = \max_{\lvert x \rvert = 1, x \in \mathbb{R}^d }\big(- (a \cdot x)^2 \big)
    = 0 =\frac{1+ a \cdot b}{2}, 
\end{align*}
which is achieved by $x \in \langle a \rangle^\perp $.

2) If $a+b \neq 0$, for $x \in \langle a,b\rangle$, set $a\cdot b = \cos \alpha$ and $a\cdot x = \cos \theta$, with $0\leq \alpha<\pi$, $0\leq \theta \leq \pi$. Then $b\cdot x = \cos ( \theta \pm \alpha)$ and 
\begin{align*}
    (a \cdot x) (b \cdot x) = \cos \theta \cos ( \theta \pm \alpha) = \frac{1}{2} \big(\cos ( 2\theta \pm \alpha) + \cos \alpha \big).
\end{align*}
It follows that 
\begin{align*}
    \max_{\lvert x \rvert= 1, x \in \langle a,b\rangle }((a \cdot x) (b \cdot x) ) = \frac{1}{2} (1 + \cos \alpha ) =  \frac{1+ a \cdot b}{2}>0,
\end{align*}
achieved by $\theta =  \frac{\alpha}{2} \mbox{ or } \pi - \frac{\alpha }{2}$, in other words, $x = \pm \frac{a+b}{\lvert a+b \rvert}$.

For $\bar x \notin \langle a,b\rangle$, $\rvert \bar x \lvert = 1$, set $\bar x = cy + z$, with $y \in \langle a,b\rangle, \lvert y \rvert =1 $, $z \in \langle a,b\rangle^\perp $, and $0\leq c<1$. We have 
\begin{align*}
    (a \cdot \bar x) (b \cdot \bar x) = c^2 (a \cdot y) (b \cdot y) \le
    c^2\frac{1+ a \cdot b}{2}<\frac{1+ a \cdot b}{2},
\end{align*}
and the assertions follow immediately.
\end{proof}

\end{appendices}
\end{document}